\documentclass[letterpaper,11pt]{article}
\usepackage[margin=1in]{geometry}
\usepackage{amsmath,amsthm,amssymb,amsfonts}
\usepackage[dvipdfmx]{graphicx}
\usepackage[capitalize]{cleveref}
\usepackage{enumitem}
\usepackage{color}
\usepackage{bbm}
\usepackage[subrefformat=parens]{subcaption}
\usepackage[normalem]{ulem}
\usepackage{xspace}
\usepackage{ifthen}
\newif\ifdraft
\drafttrue

\newcommand{\btwo}{best-of-two\xspace}

\newcommand{\bthree}{best-of-three\xspace}

\newcommand{\bk}{best-of-$k$\xspace}
\newcommand{\Bk}{Best-of-$k$\xspace}

\newcommand{\Maj}{Majority\xspace}
\newcommand{\qmaj}{quasi-majority\xspace}
\newcommand{\Qmaj}{Quasi-majority\xspace}
\newcommand{\fvoting}{functional voting\xspace}
\newcommand{\Fvoting}{Functional voting\xspace}

\renewcommand\Pr{\mathop{\mathbf{Pr}}}
\DeclareMathOperator{\E}{\mathbf{E}}
\newcommand{\Var}{\mathop{\mathbf{Var}}}
\def\defeq{\mathrel{\mathop:}=}
\def\prodp<#1>{\left \langle #1\right \rangle_{\pi}}

\newcommand\cons{\mathrm{cons}}

\newtheorem{theorem}{Theorem}[section]
\newtheorem{lemma}[theorem]{Lemma}
\newtheorem{corollary}[theorem]{Corollary}
\newtheorem{definition}[theorem]{Definition}

\newtheorem{claim}[theorem]{Claim}

\newtheorem{remark}[theorem]{Remark}

\crefname{equation}{}{}
\crefname{assumption}{Assumption}{Assumption}
\crefname{figure}{Figure}{Figure}
\crefname{enumi}{}{}
\creflabelformat{enumi}{#2#1#3}

\title{Quasi-majority Functional Voting on Expander Graphs}
\author{Nobutaka Shimizu\thanks{The University of Tokyo, Japan. {\ttfamily nobutaka\_shimizu@mist.i.u-tokyo.ac.jp}} \and Takeharu Shiraga\thanks{Chuo University, Japan. {\ttfamily shiraga.076@g.chuo-u.ac.jp}}}
\date{\today}

\begin{document}

\maketitle

\begin{abstract}
Consider a distributed graph where each vertex holds one of two distinct opinions.
In this paper, we are interested in synchronous \emph{voting processes} where each vertex updates its opinion according to a predefined common local updating rule.
For example, each vertex adopts the majority opinion among
1)~itself and two randomly picked neighbors in \emph{best-of-two} or
2)~three randomly picked neighbors in \emph{best-of-three}. 
Previous works intensively studied specific rules including best-of-two and best-of-three individually.

In this paper, 
we generalize and extend previous works of best-of-two and best-of-three on expander graphs
by proposing a new model, \emph{quasi-majority functional voting}.
This new model contains best-of-two and best-of-three as special cases.
We show that, on expander graphs with sufficiently large initial bias, any quasi-majority functional voting
reaches consensus within $O(\log n)$ steps with high probability.  
Moreover, we show that, for any initial opinion configuration, any quasi-majority functional voting  on expander graphs with higher expansion (e.g., Erd\H{o}s-R\'enyi graph $G(n,p)$ with $p=\Omega(1/\sqrt{n})$) reaches consensus within $O(\log n)$ with high probability. 
Furthermore, we show that the consensus time is $O(\log n/\log k)$ of best-of-$(2k+1)$ for $k=o(n/\log n)$.

\ 

\end{abstract}

\section{Introduction}
Consider an undirected graph $G=(V,E)$ where each vertex $v\in V$ initially holds an opinion $\sigma\in \Sigma$ from a finite set $\Sigma$.
In \emph{synchronous voting process} (or simply, \emph{voting process}), in each round, every vertex communicates with its neighbors and then all vertices simultaneously update their opinions according to a predefined protocol. 
The aim of the protocol is to reach a \emph{consensus configuration}, i.e., a configuration where all vertices have the same opinion.
Voting process has been extensively studied in several areas including biology, network analysis, physics and distributed computing~\cite{CMP09,MNT14,Liggett85,FLM86,GK10,AABHBB11}.
For example, in distributed computing, voting process plays an important role in the consensus problem~\cite{FLM86,GK10}.

This paper is concerned with the {\em consensus time} of voting processes over {\em binary} opinions $\Sigma=\{0,1\}$.
Then voting processes have state space $2^V$.
A state of $2^V$ is called a \emph{configuration}.
The \emph{consensus time} is the number of steps needed to reach a consensus configuration.


\subsection{Previous works of specific updating rules}

In \emph{pull voting}, in each round, every vertex adopts the opinion of a randomly selected neighbor. 
This is one of the most basic voting process, which has been well explored in the past~\cite{NIY99,HP01, CEOR13, CR16, BGKMT16}.
In particular, the expected consensus time of this process has been extensively studied in the literature.
For example, Hassin and Peleg~\cite{HP01} showed that the expected consensus time is $O(n^3\log n)$ for all non-bipartite graphs and all initial opinion configurations, where $n$ is the number of vertices.
From the result of Cooper, Els\"{a}sser, Ono, and Radzik~\cite{CEOR13}, it is known that on the complete graph $K_n$, the expected consensus time is $O(n)$ for any initial opinion configuration.


In \emph{\btwo} (a.k.a.~\emph{2-Choices}), each vertex $v$ samples two random neighbors (with replacement) and, if both hold the same opinion, $v$ adopts the opinion.
Otherwise, $v$ keeps its own opinion.
Doerr, Goldberg, Minder, Sauerwald, and Scheideler~\cite{DGMSS11} showed that, on the complete graph $K_n$, the consensus time of \btwo is $O(\log n)$ with high probability\footnote{In this paper ``with high probability'' (w.h.p.) means probability at least $1-n^{-c}$ for a constant $c>0$.} for an arbitrary initial opinion configuration.
Since \btwo is simple and is faster than pull voting on the complete graphs, this model gathers special attention in distributed computing and related area~\cite{GL18,CER14,CERRS15,CRRS17,CNNS18,CNS19,SS19}.
%
There is a line of works that study \btwo on expander graphs~\cite{CER14,CERRS15,CRRS17}, which we discuss later.

In \emph{\bthree} (a.k.a.~\emph{3-Majority}), each vertex $v$ randomly selects three random neighbors (with replacement).
Then, $v$ updates its opinion to match the majority among the three.
It follows directly from Ghaffari and Lengler~\cite{GL18} that, on $K_n$ with any initial opinion configuration, the consensus time of \bthree is $O(\log n)$ w.h.p.
Kang and Rivera~\cite{KR19}
considered the consensus time of
\bthree on
graphs with large minimum degree
starting from a random initial
configuration.
Shimizu and Shiraga~\cite{SS19} showed that, for any initial configurations, \btwo and \bthree reach consensus in $O(\log n)$ steps w.h.p.~if the graph is an Erd\H{o}s-R\'enyi graph $G(n,p)$\footnote{Recall that the Erd\H{o}s-R\'enyi random graph $G(n,p)$ is a graph on $n$ vertices where each of possible $\binom{n}{2}$ vertex pairs forms an edge with probability $p$ independently.} of $p=\Omega(1)$.

\emph{\Bk} ($k\geq 1$) is a generalization of pull voting, \btwo and \bthree.
In each round, every vertex $v$ randomly selects $k$ neighbors (with replacement) and then if at least $\lfloor k/2 \rfloor+1$ of them have 
the same opinion, the vertex $v$ adopts it.
Note that the best-of-$1$ is equivalent to pull voting.
Abdullah and Draief~\cite{AD15} studied
a variant of 
\bk ($k\geq 5$ is odd) 
on a specific class
of sparse graphs
that includes $n$-vertex random $d$-regular graphs\footnote{An $n$-vertex random $d$-regular graph $G_{n,d}$ is a graph selected uniformly at random from the set of all labelled $n$-vertex $d$-regular graphs.} $G_{n,d}$
of $d=o(\sqrt{\log n})$ with a random initial configuration.
To the best of our knowledge,
\bk has not been studied explicitly
so far.

In \emph{Majority} (a.k.a.~\emph{local majority}), each vertex $v$ updates its opinion to match the majority opinion among the neighbors.
This simple model has been extensively studied in previous works~\cite{BCOTT16,Berger01,GZ18,Peleg98,Peleg02,Zehmakan18}.
For example, \Maj on certain families of graphs including the Erd\H{o}s-R\'enyi random graph~\cite{BCOTT16,Zehmakan18}, random regular graphs~\cite{GZ18}
have been investigated.
See~\cite{Peleg02} for further details.

\paragraph*{Voting process on expander graphs.} 
Expander graph gathers special attention in the context of Markov chains on graphs,
yielding a wide range of
theoretical applications.
A graph $G$ is $\lambda$-\emph{expander}
if $\max\{|\lambda_2|,|\lambda_n|\}\leq \lambda$,
where $1=\lambda_1\geq \lambda_2\geq\cdots \geq \lambda_n\geq -1$
are the eigenvalues of
the transition matrix $P$ 
of the simple random walk on $G$.
For example, an Erd\H{o}s-R\'enyi graph $G(n,p)$ of $p\geq (1+\epsilon)\frac{\log n}{n}$ for an arbitrary constant $\epsilon>0$ is $O(1/\sqrt{np})$-expander w.h.p.~\cite{C07}.
An $n$-vertex random $d$-regular graph $G_{n,d}$ of $3 \leq d \leq n/2$ is $O(1/\sqrt{d})$-expander~w.h.p.~\cite{CGJ18,TY19}.

Cooper~et~al.~\cite{CEOR13} showed that the expected consensus time of pull voting is $O(n/(1-\lambda))$ on $\lambda$-expander regular graphs for any initial configuration.
Compared to pull voting, the study of \btwo on general graphs seems much harder.
Most of the previous works concerning \btwo on expander graphs put some assumptions on the initial configuration.
Let $A$ denote the set of vertices of opinion $0$ and $B=V\setminus A$.
Cooper, Els\"asser, and Radzik~\cite{CER14} showed that, for any regular $\lambda$-expander graph, the consensus time is $O(\log n)$ w.h.p.~if $\bigl||A|-|B|\bigr|=\Omega(\lambda n)$.
This result was improved by Cooper, Els\"asser, Radzik, Rivera, and Shiraga~\cite{CERRS15}.
Roughly speaking, they proved that, on $\lambda$-expander graphs, the consensus time is $O(\log n)$ if 
$|d(A)-d(B)|=\Omega(\lambda^2 d(V))$, where $d(S)=\sum_{v\in S}\deg(v)$ denotes the
volume of $S\subseteq V$.
To the best of our knowledge, the worst case consensus time of \bk on expander graphs has not been studied.


\subsection{Our model}
In this paper, we propose a new class \emph{\fvoting} of voting process,
which contains many known voting processes as a special case.
Let $A\subseteq V$ be the set of vertices of opinion $0$ and $A'$ be the set in the next round.
Let $B=V\setminus A$ and $B'=V\setminus A'$.
For $v\in V$ and $S\subseteq V$, let $N(v)=\{w\in V:\{v,w\}\in E\}$ and $\deg_S(v)=|N(v)\cap S|$.

\begin{definition}[\Fvoting] \label{def:fVoting}
Let $f:\mathbb{R}\to \mathbb{R}$ be
a function satisfying $f([0,1])=[0,1]$ and $f(0)=0$.
A \emph{\fvoting with respect to $f$} is a synchronous voting process defined~as
\begin{align*}
\begin{aligned}
&\Pr[v\in A'] =
f\left(\frac{\deg_A(v)}{\deg(v)}\right)
\quad \text{if $v\in B$}, \\
&\Pr[v\in B'] = f\left(\frac{\deg_{B}(v)}{\deg(v)}\right) \quad \text{if $v\in A$}.
\end{aligned}
\end{align*}
We call the function $f$ a \emph{betrayal function} and the function 
\begin{align*}
    H_f(x)\defeq x\bigl(1-f(1-x)\bigr)+(1-x)f(x)
\end{align*}
an \emph{updating function}.
\end{definition}
Since $f(0)=0$, consensus configurations are absorbing states.
Hence the consensus time is well-defined\footnote{For disconnected graphs, there exist initial configurations that never reach consensus.
Henceforth, we are concerned with connected graphs.}.
%
The intuition behind the updating function $H_f$ is that, 
letting $\alpha=|A|/n$ and $\alpha'=|A'|/n$, on a complete graph $K_n$ (with self-loop), 
the \fvoting with respect to $f$ satisfies $\E[\alpha']=\frac{|A|}{n}\left(1-f\left(\frac{|B|}{n}\right)\right)+\frac{|B|}{n}f\left(\frac{|A|}{n}\right)= H_f(\alpha)$.


\Fvoting contains many existing models as special cases.
For example, pull voting, \btwo, and \bthree are functional votings with respect to $x$, $x^2$ and 
$3x^2-2x^3$, respectively.
In general, \bk is a \fvoting with respect to
\begin{align}
f_k(x)=\sum_{i=\lfloor k/2 \rfloor+1}^k \binom{k}{i} x^i (1-x)^{k-i}. \label{eq:bokbetrayal}
\end{align}
It is straightforward to check that $H_{f_k}(x)=f_k(x)$ if $k$ is odd and $H_{f_k}(x)=f_{k+1}(x)$ if $k$ is even.
\Maj is a functional voting with respect to
\begin{align}
    f(x)=\begin{cases}
    0 & \text{if $x<\frac{1}{2}$},\\
    \frac{1}{2} & \text{if $x=\frac{1}{2}$},\\
    1 & \text{if $x>\frac{1}{2}$}
    \end{cases} \label{eq:majoritybetrayal}
\end{align}
if a vertex adopts the random opinion when it meets the tie.

\paragraph*{\Qmaj \fvoting.}
%
In this paper, we focus on \fvoting with respect to $f$ satisfying the following property.
\begin{definition}[\Qmaj.]
\label{def:quasi-majority}
A function $f$ is \emph{\qmaj} if $f$ satisfies the following conditions.
\begin{enumerate}[label=(\arabic*)]
    \item\label{cond:fC2} $f$ is $C^2$,
    \item\label{cond:f1/2} $0<f(1/2)<1$,
    \item\label{cond:Hfx>x} $H_f(x)<x$ whenever $x\in(0,1/2)$.
    \item\label{cond:katamuki_center} $H'_f(1/2)>1$,
    \item\label{cond:katamuki_edge}$H'_f(0)<1$.
\end{enumerate}
A voting process is a \emph{\qmaj \fvoting} if it is a \fvoting with respect to a \qmaj function $f$.
\end{definition}
%
Note that $H_f(x)$ is symmetric (i.e., $H_f(1-x)=1-H_f(x)$) and thus the condition \cref{cond:Hfx>x} implies $H_f(x)>x$ for every $x\in (1/2,1)$.
Intuitively, the conditions \cref{cond:Hfx>x} to \cref{cond:katamuki_edge} ensure the drift 
towards consensus.
The conditions~\cref{cond:fC2} and \cref{cond:f1/2} are due to a technical reasons.

\begin{figure}[t]
\centering
\includegraphics[width=8cm]{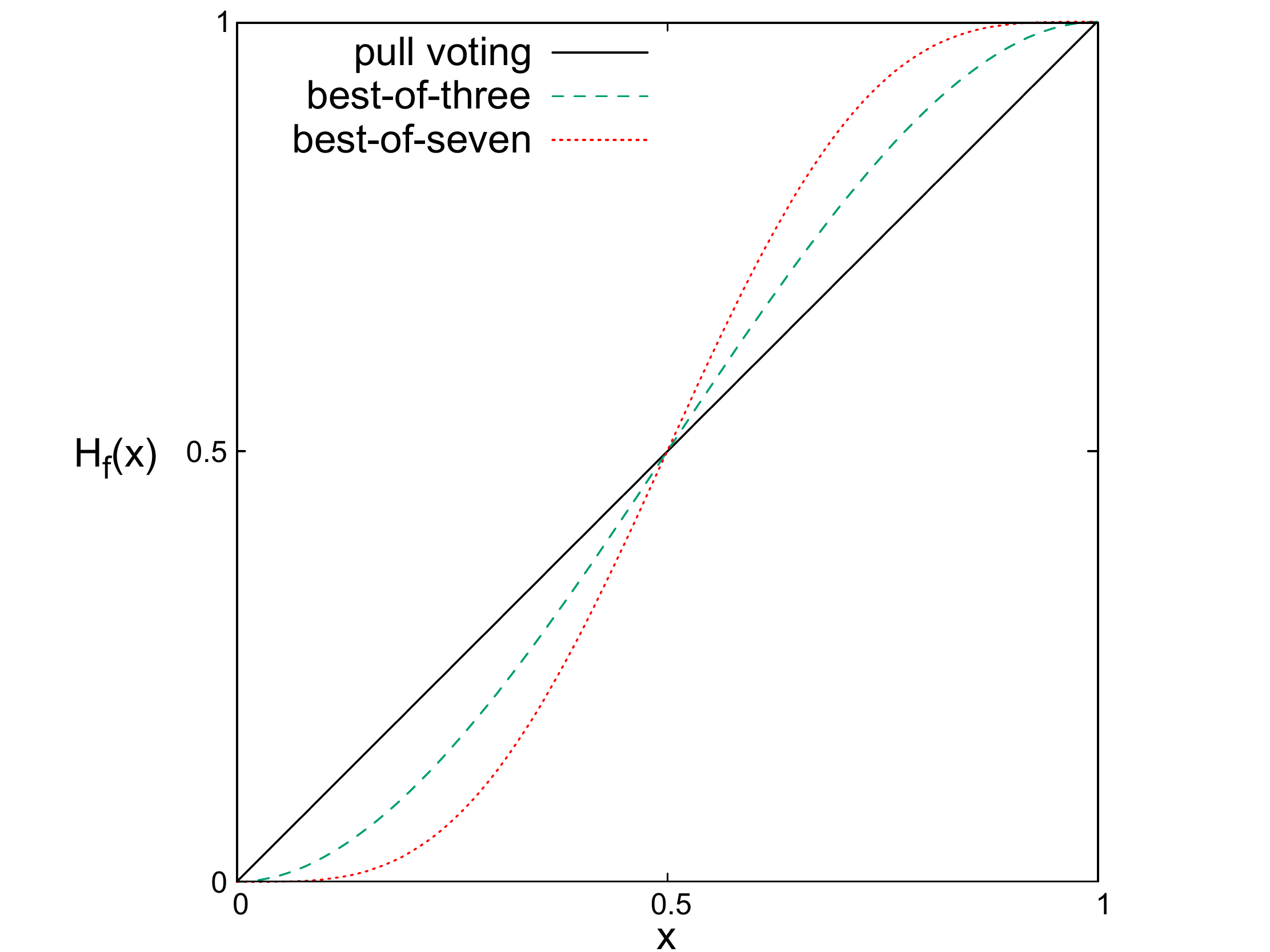}
\caption{The update functions $H_f(x)$ of pull voting (solid line), \bthree (dashed line) and best-of-seven (dotted line).
One can easily observe that \bthree and best-of-seven are \qmaj \fvoting.
Intuitively speaking, \qmaj \fvoting processes have updating functions $H_f$ with the property so-called ``the rich get richer", which coincides with \cref{def:quasi-majority}.
\label{fig:qmajplot}
}
\end{figure}

For each constant $k\geq 2$, \bk is
\qmaj \fvoting but pull voting and \Maj are not.
Indeed, if $H_{f_k}$ is the updating function of \bk, then
$
    H_{f_{2\ell}}'(x) = H_{f_{2\ell+1}}'(x) = (2\ell+1)\binom{2\ell}{\ell}x^\ell(1-x)^\ell.
$
It is straightforward to check that this function satisfies the conditions \cref{cond:Hfx>x} to \cref{cond:katamuki_edge} if $\ell\neq 0$ (pull-voting).
See \cref{fig:qmajplot} for depiction of the updating functions of pull voting, \bthree and best-of-seven.

\subsection{Our result}
In this paper, we study the consensus time of \qmaj \fvoting on expander graphs\footnote{Throughout the paper, we consider sufficiently large $n=|V|$.}.
Let $T_{\cons}(A)$ denote the consensus time starting from the initial configuration $A\subseteq V$.
For a graph $G=(V,E)$, let $\pi=(\pi(v))_{v\in V}$ denote the {\em degree distribution} defined as
\begin{align}
\pi(v)=\frac{\deg(v)}{2|E|}.\label{def:degreedist}
\end{align}
Note that $\sum_{v\in V}\pi(v)=1$ holds.
We denote by $\|x\|_p\defeq \left(\sum_{v\in V}|x_v|^p\right)^{1/p}$ the $\ell^p$ norm of $x\in \mathbbm{R}^V$.
For $\pi\in [0,1]^V$ and $A\subseteq V$, let $\pi(A)\defeq \sum_{v\in A}\pi(v)$. Let 
\begin{align*}
\delta(A)\defeq \pi(A)-\pi(V\setminus A) = 2\pi(A)-1
\end{align*}
denote the {\em bias} between $A$ and $V\setminus A$.

\begin{theorem}[Main theorem]
\label{thm:mainthm}
Consider a \qmaj \fvoting with respect to~$f$ on an $n$-vertex $\lambda$-expander graph with degree distribution $\pi$.
Then, the following holds:
\begin{enumerate}[label=(\roman*)]
    \item \label{thm:worstcasefvoting} Let $C_1>0$ be an arbitrary constant and $\varepsilon:\mathbb{N}\to\mathbb{R}$ be an arbitrary function satisfying $\varepsilon(n)\to 0$ as $n\to \infty$.
    Suppose that $\lambda \leq C_1n^{-1/4}$, $\|\pi\|_2 \leq C_1/\sqrt{n}$ and $\|\pi\|_3\leq \varepsilon / \sqrt{n}$.
    Then, for any $A\subseteq V$, $T_\cons(A)=O(\log n)$ w.h.p.
    \item \label{thm:cons_initial_bias} Let $C_2$ be a positive constant depending only on $f$.
    Suppose that $\lambda\leq C_2$ and $\|\pi\|_2\leq C_2/\sqrt{\log n}$. 
    Then, for any $A\subseteq V$ satisfying $|\delta(A)|\geq C_2\max\{\lambda^2,\|\pi\|_2\sqrt{\log n}\}$, $T_{\cons}(A)=O(\log n)$ w.h.p. 
\end{enumerate}
\end{theorem}
The following result which we show in \cref{app:lower} indicates that the consensus time of \cref{thm:mainthm}\cref{thm:worstcasefvoting} is optimal up to a constant factor. 
\begin{theorem}[Lower bound] \label{thm:lowerboundconsensustime}
Under the same assumption of \cref{thm:mainthm}\cref{thm:worstcasefvoting}, $T_{\cons}(A) = \Omega(\log n)$ w.h.p.~for some $A\subseteq V$.
\end{theorem}
\begin{theorem}[Fast consensus for $H_f'(0)=0$]
\label{thm:fastconsensus}
Consider a \qmaj \fvoting with respect to $f$ on an $n$-vertex $\lambda$-expander graph with degree distribution $\pi$.
Let $C>0$ be a constant depending only on $f$. 
Suppose that $H_f'(0)=0$, $\lambda\leq C$ and $\|\pi\|_2\leq C/\sqrt{\log n}$. 
Then, for any $A\subseteq V$ satisfying $|\delta(A)|\geq C\max\{\lambda^2,\|\pi\|_2\sqrt{\log n}\}$, it holds w.h.p.~that
$$T_{\cons}(A)=O\left(\log \log n +\log |\delta(A)|^{-1} + \frac{\log n}{\log \lambda^{-1}}+ \frac{\log n}{\log (|\pi\|_2\sqrt{\log n})^{-1}}\right).$$ 
\end{theorem}
For example, for each constant $k\geq 2$, best-of-$k$ is quasi-majority with $H_f'(0)=0$.
\begin{remark}
Roughly speaking, for $p\geq 2$, $\|\pi\|_p$ measures the imbalance of the degrees.
For any graphs, $\|\pi\|_p\geq n^{-1+1/p}$ and the equality holds
if and only if the graph is regular.
For star graphs, we have $\|\pi\|_p\approx 1$.
\end{remark}
\paragraph*{Results of \bk.}
Our results above do not explore \Maj since it is not \qmaj.
A plausible approach is to consider \bk for $k=k(n)=\omega(1)$ since
each vertex is likely to choose
the majority opinion if the number of neighbor sampling increases.
Also, note that
the betrayal function $f_k$ of \bk given in \cref{eq:bokbetrayal} converges to that of \Maj (i.e., $f_k(x)\to f(x)$ as $k\to\infty$ for each $x\in[0,1]$, where $f$ is the betrayal function~\cref{eq:majoritybetrayal} of Majority).
On the other hand, if $k=O(1)$, there is a tremendous gap between \bk and \Maj:
For any \fvoting on the complete graph $K_n$, 
$T_{\cons}(A)=\Omega(\log n)$ for some $A\subseteq V$~from \cref{thm:lowerboundconsensustime}.
\Maj on $K_n$ reaches the consensus in a single step if $|A|<|V\setminus A|-1$.
This motivates us to consider \bk for $k=k(n)\to \infty$ as $n\to \infty$.
For simplicity,
we focus on best-of-$(2k+1)$ and prove the following result in \cref{sec:bokproof}.

\begin{theorem} \label{thm:bok}
Let $k=k(n)$ be such that $k=\omega(1)$ and $k=o(n/\log n)$.
Let $C$ be an arbitrary positive constant.
Consider best-of-$(2k+1)$ on an $n$-vertex $\lambda$-expander graph with degree distribution $\pi$ such that $\lambda\leq Ck^{-1/2}n^{-1/4}$, $\|\pi\|_2\leq Cn^{-1/2}$ and $\|\pi\|_3\leq C k^{-1/6}n^{-1/2}$.
Then, $T_{\cons}(A)=O\left(\frac{\log n}{\log k}\right)$ holds w.h.p.~for any $A\subseteq V$.
\end{theorem}


\subsection{Application} \label{sec:application}
Here, we apply our main theorem to specific graphs and derive some useful results.

For any $p\geq (1+\epsilon)\frac{\log n}{n}$ for an arbitrary constant $\epsilon>0$, $G(n,p)$ is connected and $O(1/\sqrt{np})$-expander w.h.p~\cite{C07,FK16}.
\begin{corollary}
\label{cor:Gnp}
Consider a \bk on an Erd\H{o}s-R\'enyi graph $G(n,p)$ for an arbitrary constant $k\geq 2$.
Then, $G(n,p)$ w.h.p.~satisfies the following:
\begin{enumerate}[label=(\roman*)]
    \item \label{lab:Gnp1} Suppose that $p =\Omega(n^{-1/2})$. Then
    \begin{enumerate}
        \item for any $A\subseteq V$, $T_\cons(A)=O(\log n)$ w.h.p.
        \item for some $A\subseteq V$, $T_\cons(A)=\Omega(\log n)$ w.h.p.
    \end{enumerate}
    \item \label{lab:Gnp2} 
    Suppose that $p\geq (1+\epsilon)\frac{\log n}{n}$ for an arbitrary constant $\epsilon>0$.
    Then, for any $A\subseteq V$ satisfying $|\delta(A)|\geq C\max\Bigl\{\frac{1}{np},\sqrt{\frac{\log n}{n}}\Bigr\}$, 
    $T_{\cons}(A)=O\left(\log \log n +\log |\delta(A)|^{-1} + \frac{\log n}{\log (np)} \right)$ w.h.p., 
    where $C>0$ is a constant depending only on $f$.
\end{enumerate}
\end{corollary}
In \cref{cor:Gnp}\cref{lab:Gnp1}, we stress that the worst-case consensus time on $G(n,p)$ was known for $p=\Omega(1)$~\cite{SS19}.
If $\frac{\log n}{\log (np)}=O(\log \log n)$ (or equivalently, $np=n^{\Omega(1/\log \log n)}$), 
\cref{cor:Gnp}\cref{lab:Gnp2} implies
$T_{\cons}(A)=O(\log \log n+\log |\delta(A)|^{-1})$ w.h.p.
%
\begin{corollary}
\label{cor:bokGnp}
Let $k=k(n)$ be such that $k=\omega(1)$ and $k=O(\sqrt{n})$. 
Consider best-of-$(2k+1)$ on $G(n,p)$ for $p=\Omega(k/\sqrt{n})$.
Then, for any $A\subseteq V$, $T_{\cons}(A)=O\left(\frac{\log n}{\log k}\right)$ holds w.h.p.
\end{corollary}
From \mbox{\cref{cor:bokGnp}}, best-of-$n^{\epsilon}$ on $G(n,n^{-1/2+\epsilon})$ for any constant $\epsilon\in(0,1/2)$ reaches consensus in $O(1)$ steps.
It is known that \Maj on $G(n,Cn^{-1/2})$ satisfies $T_{\mathrm{cons}}(A)\leq 4$ for large constant $C$ and random $A\subseteq V$ with constant probability~\cite{BCOTT16}.

For $3 \leq d \leq n/2$, $n$-vertex random $d$-regular graph $G_{n,d}$ is
connected and
$O(1/\sqrt{d})$-expander w.h.p.~\cite{CGJ18,TY19}. 
\begin{corollary}
\label{cor:Gnd}
Consider a \bk on an $n$-vertex random $d$-regular graph $G_{n,d}$ for an arbitrary constant $k\geq 2$.
Then, $G_{n,d}$ w.h.p.~satisfies the following:
\begin{enumerate}[label=(\roman*)]
    \item 
    Suppose that $d =\Omega(n^{1/2})$ and $d\leq n/2$. Then, 
    \begin{enumerate}
        \item for any $A\subseteq V$, $T_\cons(A)=O(\log n)$ w.h.p.
        \item for some $A\subseteq V$, $T_\cons(A)=\Omega(\log n)$ w.h.p.
    \end{enumerate}
    \item 
    Suppose that $d\geq C$ and $d\leq n/2$ for a constant $C>0$ depending only on $f$.
    Then, for any $A\subseteq V$ satisfying $|\delta(A)|\geq C\max\Bigl\{\frac{1}{d},\sqrt{\frac{\log n}{n}}\Bigr\}$, it holds w.h.p.~that 
    $T_{\cons}(A)=O\left(\log \log n +\log |\delta(A)|^{-1} + \frac{\log n}{\log d} \right)$.
\end{enumerate}
\end{corollary}
\begin{corollary}
\label{cor:bokGnd}
Let $k=k(n)$ be such that $k=\omega(1)$ and $k=O(\sqrt{n})$.
Consider best-of-$(2k+1)$ on an $n$-vertex random $d$-regular graph $G_{n,d}$ such that $d=\Omega(k\sqrt{n})$ and $d\leq n/2$.
Then, for any $A\subseteq V$, $T_{\cons}(A)=O\left(\frac{\log n}{\log k}\right)$ holds w.h.p.
\end{corollary}
%
We can apply \cref{thm:mainthm,thm:fastconsensus} if the ratio of the maximum and average degree is constant as follows.
\begin{corollary}
\label{cor:avedegree}
Consider a \qmaj \fvoting with respect to $f$ on an $n$-vertex $\lambda$-expander graph with degree distribution $\pi$. 
Suppose that $d_{\max}\leq C_1d_{{\rm ave}}$ for an arbitrary constant $C_1>0$,
where $d_{\max}$ and $d_{\mathrm{avr}}$ denote the maximum and average degree, respectively.
Then, the following holds:
\begin{enumerate}[label=(\roman*)]
    \item 
    Suppose that $\lambda \leq C_1n^{-1/4}$. Then
    \begin{enumerate}
        \item for any $A\subseteq V$, $T_\cons(A)=O(\log n)$ w.h.p.
        \item for some $A\subseteq V$, $T_\cons(A)=\Omega(\log n)$ w.h.p.
    \end{enumerate}
    \item 
    Suppose that $\lambda\leq C_2$ for some constant $C_2>0$ depending only on $f$.
    Then, for any $A\subseteq V$ satisfying $|\delta(A)|\geq C_2\max\bigl\{\lambda^2,\sqrt{\frac{\log n}{n}}\bigr\}$, $T_{\cons}(A)=O(\log n)$ w.h.p.
    \item In addition to the same assumption as (ii), suppose that
    $H_f'(0)=0$.
    Then, it holds w.h.p.~that
    $T_{\cons}(A)=O\left(\log \log n +\log |\delta(A)|^{-1} + \frac{\log n}{\log \lambda^{-1}} \right)$.
\end{enumerate}
\end{corollary}
\begin{corollary}
\label{cor:bok_avedegree}
Let $k=k(n)$ be such that $k=\omega(1)$ and $k=o(n/\log n)$.
Let $C$ be an arbitrary constant.
Consider best-of-$(2k+1)$ on an $n$-vertex $\lambda$-expander graph with degree distribution $\pi$ such that $\lambda\leq Ck^{-1/2}n^{-1/4}$, and $d_{\max}\leq C d_{\mathrm{avr}}$, where $d_{\max}$ and $d_{\mathrm{avr}}$ denote the maximum and average degree, respectively.
Then, $T_{\cons}(A)=O\left(\frac{\log n}{\log k}\right)$ holds
w.h.p.~for any $A\subseteq V$.
\end{corollary}

\Cref{cor:avedegree,cor:bok_avedegree} immediately follow from \cref{thm:mainthm,thm:lowerboundconsensustime,thm:fastconsensus,thm:bok} since $\|\pi\|_2 =O(n^{-1/2})$.
Note that if the ratio of the maximum degree $d_{\max}$ and average degree $d_{\mathrm{avr}}$ is constant,
$\|\pi\|_p=\Theta(1/n^{1-1/p})$ since $\pi(v)=O(1/n)$ for all $v\in V$. 
We obtain \cref{cor:Gnp,cor:bokGnp,cor:Gnd,cor:bokGnd} from \cref{cor:avedegree,cor:bok_avedegree}. 

\paragraph*{Other \qmaj \fvoting.}
We can consider the \emph{$\rho$-lazy} variant of a voting process, i.e., every vertex $v$ individually tosses its private coin and operates the voting process with probability $\rho$, while $v$ does nothing with probability $1-\rho$.
Berenbrink, Giakkoupis, Kermarrec, and Mallmann-Trenn~\cite{BGKMT16} studies $1/2$-lazy pull voting.
If the original voting process is a \qmaj \fvoting with respect to $f$, then the corresponding $\rho$-lazy variant is \qmaj functional voting with respect to $\rho f:x\mapsto \rho f(x)$.
Indeed, $H_{\rho f}(x) = (1-\rho)x+\rho H_f(x)$.

\begin{corollary}
Consider a $\rho$-lazy \qmaj \fvoting on $G(n,p)$ for an arbitrary constant $\rho\in (0,1]$. Suppose that $p=\Omega(1/\sqrt{n})$.
Then, for any $A\subseteq V$, $T_{\cons}(A)=O(\log n)$ w.h.p.
\end{corollary}
This implies the following interesting observation.
In voting processes, the number of neighbor sampling queries per each vertex at each step affects the performance.
In pull voting, each vertex communicates with one neighbor but it has a drawback on the slow consensus time.
In \btwo, each vertex communicates with two random neighbors and its consensus time is much faster than that of pull voting.
In $\rho$-lazy \btwo, each vertex queries $2\rho$ vertices at each round in expectation, that is less queries than pull voting if $\rho<1/2$. 
On the other hand, the consensus time is much faster than pull voting.

Additionally, we can deal with {\em $k$-careful voting}.
In this model, each vertex $v$ selects $k$ random neighbors (with replacement), and if these sampled $k$ opinions are the same one, $v$ adopts it.
Note that one-careful voting and two-careful voting are equivalent to pull voting and \btwo, respectively.
One can check easily that, for any constant $k\geq 2$, this model is a \qmaj \fvoting with respect to
$f(x)=x^k$.
Note that 
$H_f'(0)=0$ and $H_f'(1/2)=1+\frac{k-1}{2^{k-1}}$.


\begin{corollary}
Consider a $k$-careful voting on $G(n,p)$ for an arbitrary constant $k\geq 2$. Suppose that $p=\Omega(1/\sqrt{n})$.
Then, for any $A\subseteq V$, $T_{\cons}(A)=O(\log n)$ w.h.p.
\end{corollary}

\subsection{Related work}
In asynchronous voting process,
in each round,
a vertex is selected uniformly at random
and only the selected vertex updates its opinion.
Cooper and Rivera~\cite{CR16}
introduced \emph{linear voting model}.
In this model, an opinion
configuration
is represented as a vector $v\in \Sigma^V$ and the
vector $v$ updates according to the rule $v\leftarrow Mv$, where
$M$ is a random matrix sampled from
some probability space.
This model captures a wide variety model
including asynchronous push/pull voting and synchronous pull voting.
Note that \btwo and \bthree are not included in linear voting model.
Schoenebeck and Yu~\cite{SY18}
proposed an asynchronous variant of our \fvoting.
The authors of \cite{SY18} proved that, if the function $f$ is symmetric (i.e., $f(1-x)=1-f(x)$), smooth and has ``majority-like'' property (i.e., $f(x)>x$ whenever $1/2<x<1$), then the expected consensus time is $O(n\log n)$ w.h.p.~on $G(n,p)$ with $p=\Omega(1)$.
This perspective has also been investigated in physics~(see, e.g.,~\cite{CMP09}).

Several researchers have studied \btwo and \bthree on complete graphs initially involving $k\geq 2$ opinions~\cite{BCNPT16, BCNPST17, BCEKMN17, GL18}.
For example, the consensus time of \bthree is $O(k\log n)$ if $k=O(n^{1/3}/\sqrt{\log n})$~\cite{GL18}.
Cooper, Radzik, Rivera, and Shiraga~\cite{CRRS17} considered \btwo and \bthree on regular expander graphs that hold more than two opinions.

Recently, Cruciani, Natale, and Scornavacca~\cite{CNS19} 
studied 
\btwo with a random initial configuration on a clustered regular graph.
Shimizu and Shiraga~\cite{SS19} obtained phase-transition results of \btwo and \bthree on stochastic block models.

\section{Preliminary and technical result}
\subsection{Formal definition}
Let $G=(V,E)$ be an undirected and connected graph.
Let $P\in[0,1]^{V\times V}$ be the matrix defined as
\begin{align}
P(u,v)&\defeq \frac{\mathbbm{1}_{\{u,v\}\in E}}{\deg(u)} \ \ \ \forall (u,v)\in V\times V
\label{def:SRW}
\end{align}
where $\mathbbm{1}_{Z}$ denotes the indicator
of an event $Z$.
For $v\in V$ and $S\subseteq V$,
we write $P(v,S)=\sum_{s\in S}P(v,s)$.

Now, let us describe the formal definition of \fvoting.
For a given $A\subseteq V$, let $(X_v)_{v\in V}$ be independent
binary random variables defined as
\begin{align}
\begin{aligned}
\Pr[X_v=1]=f\bigl(P(v,A)\bigr) \quad \text{if $v\in B$}, \\ \Pr[X_v=0]=f\bigl(P(v,B)\bigr)  \quad \text{if $v\in A$},
\end{aligned}
\label{eq:Xvdef}
\end{align}
where $B=V\setminus A$.
For $A\subseteq V$ and $(X_v)$ above, define $A'=\{v\in V:X_v=1\}$.
Note that this definition coincides with \cref{def:fVoting} since $P(v,A)=\frac{\deg_A(v)}{\deg(v)}$.
Then, a \fvoting is a Markov chain $A_0,A_1,\ldots$ where $A_{t+1}=(A_t)'$.

For $A\subseteq V$, let $T_{\cons}(A)$ denote the consensus time of the \fvoting starting from the initial configuration $A$.
Formally, $T_{\cons}(A)$ is the stopping time defined as
\begin{align*}
    T_{\cons}(A)\defeq \min\left\{t\geq 0: A_t\in \{\emptyset,V\}, A_0=A\right\}.
\end{align*}

\subsection{Technical background} \label{sec:technicalbackground}
Consider  
\btwo
on a complete graph $K_n$ (with self loop on each vertex) with a current configuration $A\subseteq V$.
Let $\alpha=|A|/n$.
%
We have $P(v,A)=\alpha$ for any $v\in V$ and $A\subseteq V$.
Then, for any $A\subseteq V$, $\E[\alpha']=H_f(\alpha)=3\alpha^2-2\alpha^3$. 
Thus, in each round, $\alpha'= 3\alpha^2-2\alpha^3 \pm O(\sqrt{\log n/n})$ holds w.h.p.~from the Hoeffding bound.
Therefore, the behavior of $\alpha$ can be written as the iteration of applying $H_f$.

The most technical part is the symmetry breaking at $\alpha=1/2$.
Note that $H_f(1/2)=1/2$ and thus, the argument above does not work in the case of $|\alpha-1/2|=o(\sqrt{\log n/n})$.
To analyze this case, the authors of \cite{DGMSS11,CGGNPS18} proved the following
technical lemma asserting that $\alpha$ w.h.p.~escapes from
the area in $O(\log n)$ rounds.
\begin{lemma}[Lemma 4.5 of \cite{CGGNPS18} (informal)] \label{lem:nazolemma_informal}
For any constant $C$, it holds w.h.p.~that $|\alpha-1/2|\geq C\sqrt{\log n/n}$ in $O(\log n)$ rounds (the hidden constant factor depends on $C$) if
\begin{enumerate}[label=(\roman*)]
\item \label{cond:nazo_cond1_informal}
For any constant $h$, there is a constant $C_0>0$ such that, if $|\alpha-1/2|=O(\sqrt{\log n/n})$ then $\Pr[|\alpha'-1/2|> h/\sqrt{n}] > C_0$.

\item \label{cond:nazo_cond2_informal}
If $|\alpha-1/2|=O(\sqrt{\log n/n})$ and $|\alpha-1/2|=\Omega(1/\sqrt{n})$,
$\Pr[|\alpha'-1/2| \leq (1+\epsilon) |\alpha-1/2|] \leq \exp(-\Theta((\alpha-1/2)^2n))$  
for some constant $\epsilon>0$.
\end{enumerate}
\end{lemma}

Intuitively speaking,
the condition \cref{cond:nazo_cond2_informal} means that the bias $|\alpha'-1/2|$ is likely to be at least $(1+\epsilon)|\alpha-1/2|$
for some constant $\epsilon>0$.
The condition \cref{cond:nazo_cond2_informal}
is easy to check using the Hoeffding bound.
The condition \cref{cond:nazo_cond1_informal} means that $\alpha'$ has a fluctuation of size $\Omega(1/\sqrt{n})$ with a constant probability.
We can check condition \cref{cond:nazo_cond1_informal}
using the Central Limit Theorem (the Berry-Esseen bound, see~\cref{lem:Berry-Esseen}).
The Central Limit Theorem
implies that the normalized random variable
$(\alpha'-\E[\alpha'])/\sqrt{\Var[\alpha']}$
converges to the standard normal distribution as $n\to\infty$.
In other words, $\alpha'$ has
a fluctuation of size $\Theta(\sqrt{\Var[\alpha']})$ with constant probability.
Now, to verify the condition \cref{cond:nazo_cond1_informal}, we evaluate
$\Var[\alpha']$.
On $K_n$, it is easy to show that $\Var[\alpha']=\Theta(1/n)$, which implies the condition~\cref{cond:nazo_cond1_informal}.

The authors of~\cite{CERRS15,CRRS17} considered \btwo on expander graphs.
They focused on the behavior of $\pi(A)$ instead of $\alpha$. 
Roughly speaking, they proved that 
$\E[\pi(A')-1/2]\geq (1+\epsilon)(\pi(A)-1/2)-O(\lambda^2)$.
At the heart of the proof, they showed the following result.
\begin{lemma}[Special case of Lemma 3 of \cite{CRRS17}]
\label{lem:square2}
Consider a $\lambda$-expander graph with degree distribution $\pi$.
Then, for any $S\subseteq V$,
\begin{align*}
\left|\sum_{v\in V}\pi(v)P(v,S)^2-\pi(S)^2\right|\leq \lambda^2\pi(S)\bigl(1-\pi(S)\bigr).
\end{align*}
\end{lemma}
Then, from the Hoeffding bound, we have
$\E[\pi(A')-1/2]\geq  (1+\epsilon)(\pi(A)-1/2)-O(\lambda^2+\|\pi\|_2\sqrt{\log n}))$.
Thus, if the initial bias $|\pi(A)-1/2|$ is $\Omega(\max\{\lambda^2,\sqrt{\log n/n}\})$, we can show that the consensus time is $O(\log n)$.
 

Unfortunately, we can not apply the same technique to estimate $\Var[\pi(A')]$ on expander graphs, and due to this reason, it seems difficult to estimate the worst-case consensus time on expander graphs.
Actually, any previous works 
put assumptions on the initial bias due to the same reason.
It should be noted that
\cref{lem:nazolemma_informal}
is well-known in the literature.
For example, Cruciani et al.~\cite{CNS19} used
\cref{lem:nazolemma_informal}
from random initial configurations.

The technique of estimating $\E[\pi(A')]$ by Cooper et al.~\cite{CERRS15,CRRS17} is specialized in \btwo.
Thus, it is not straightforward
to prove the estimation
of $\E[\pi(A')]$ for
voting processes other than \btwo.

\subsection{Our technical contribution}
For simplicity, in this part, we focus on a \qmaj \fvoting with respect to a {\em symmetric} function $f$ (i.e.,  $f(1-x)=1-f(x)$ for every $x\in[0,1]$) on a $\lambda$-expander graph with degree distribution $\pi$.
For example, $f(x)=3x^2-2x^3$ of \bthree is a symmetric function.
Note that $f=H_f$ if $f$ is symmetric.
Similar results mentioned in this subsection holds for non-symmetric $f$ (see \cref{sec:estimate_nonsymmetric}).
For a $C^2$ function $h:\mathbbm{R}\to\mathbbm{R}$, let
\begin{align*}
    K_1(h)\defeq \max_{x\in [0,1]}\left|h'(x)\right|, \hspace{1em}
    K_2(h)\defeq \max_{x\in [0,1]}\left|h''(x)\right|
\end{align*}
be some constants\footnote{For example, for $f(x)=3x^2-2x^3$ of \bthree, $f''(x)=6-12x$ and $K_2(f)=6$.
It should be noted that we deal with $f$ not depending on $G$ except for \bk with $k=\omega(1)$ in \cref{sec:bokproof}.} depending only on $h$.
The following technical result enables us to estimate $\E[\pi(A')]$ and $\Var[\pi(A')]$ of \fvoting.
\begin{lemma} 
\label{lem:expectationandvariance_symmetry}
Consider a \fvoting with respect to a symmetric $C^2$ function $f$ on a $\lambda$-expander graph with degree distribution $\pi$.
Let $g(x)\defeq f(x)(1-f(x))$.
Then, for all $A\subseteq V$,
\begin{align*}
    &\bigl|\E[\pi(A')]-H_f(\pi(A)) \bigr| \leq \frac{K_2(f)}{2}\lambda^2\pi(A)\bigl(1-\pi(A)\bigr),\\
    &\Bigl|\Var[\pi(A')]-\|\pi\|_2^2g\bigl(\pi(A)\bigr)\Bigr|\leq K_1(g)\lambda\sqrt{\pi(A)\bigl(1-\pi(A)\bigr)}\|\pi\|^{3/2}_3.
\end{align*}
\end{lemma}

Note that, if $f$ is symmetric, the corresponding \fvoting satisfies that $\Pr[v\in A']=f(P(v,A))$ for any $v\in V$.
Thus we have
\begin{align*}
\E[\pi(A')]
=\sum_{v\in V}\pi(v)f\bigl(P(v,A)\bigr), \hspace{1em}
\Var[\pi(A')]
=\sum_{v\in V}\pi(v)^2g\bigl(P(v,A)\bigr).
\end{align*}
To evaluate $\E[\pi(A')]$ and $\Var[\pi(A')]$ above, we prove the following key lemma that is a generalization of \cref{lem:square2} and implies \cref{lem:expectationandvariance_symmetry}.
\begin{lemma}[Special case of \cref{lem:second_approx,lem:approx_V}]
\label{lem:gen_square}
Consider a $\lambda$-expander graph with degree distribution $\pi$.
Then, for any $S\subseteq V$ and 
 any $C^2$ function $h:\mathbb{R}\to \mathbb{R}$,
\begin{align*}
&\left|\sum_{v\in V}\pi(v)h\bigl(P(v,S)\bigr)-h\bigl(\pi(S)\bigr)\right|
 \leq \frac{K_2(h)}{2}\lambda^2\pi(S)\bigl(1-\pi(S)\bigr), \\
&\left|\sum_{v\in V}\pi(v)^2h\bigl(P(v,S)\bigr)-\|\pi\|_2^2 h\bigl(\pi(S)\bigr)\right|\leq  K_1(h)\lambda\sqrt{\pi(S)\bigl(1-\pi(S)\bigr)}\|\pi\|_3^{3/2}.
\end{align*}
\end{lemma}


\subsection[Proof sketch of Theorem 1.3]{Proof sketch of \Cref{thm:mainthm}}
We present proof sketch of \cref{thm:mainthm}\cref{thm:worstcasefvoting}.
From the assumption of \cref{thm:mainthm}\cref{thm:worstcasefvoting}
and
\cref{lem:expectationandvariance_symmetry},
if 
$|\pi(A)-1/2|=o(1)$, we have
$\Var[\pi(A')]=\Theta(\|\pi\|_2^2g(\pi(A)))=\Theta(\|\pi\|_2^2g(1/2+o(1)))=\Theta(1/n)$.
Moreover,
$\E[\pi(A')]=H_f(\pi(A))\pm O(\pi(A)/\sqrt{n})$ holds for any $A\subseteq V$. Hence, from the Hoeffding bound, 
$\pi(A')=H_f(\pi(A))+O(\sqrt{\log n/n})$ holds w.h.p.~for any $A\subseteq V$.

\begin{itemize}
    \item If $|\pi(A)-1/2| =O(\sqrt{\log n/n})$, we use \cref{lem:nazolemma_informal} to obtain an $O(\log n)$ round symmetry breaking.
    In this phase, 
   since $|\pi(A)-1/2|=o(1)$, $\Var[\pi(A')-1/2]=\Theta(1/n)$.
    Then, from the Berry-Esseen theorem (\cref{lem:Berry-Esseen}), we can check the condition \cref{cond:nazo_cond1_informal}.
    To check the condition \cref{cond:nazo_cond2_informal}, we invoke the condition $H_f'(1/2)>1$ of the \qmaj function.
    From Taylor's theorem and the assumption of \cref{lem:nazolemma_informal}\cref{cond:nazo_cond2_informal} ($\pi(A)-1/2=\Omega(1/\sqrt{n})$), $\E[\pi(A')-1/2]=H_f(\pi(A))-H_f(1/2)-O(1/\sqrt{n})\approx (1+\epsilon_1)(\pi(A)-1/2)$ for some positive constant $\epsilon_1>0$. Note that $H_f(1/2)=1/2$.

    \item If $C_1\sqrt{\log n/n}\leq |\pi(A)-1/2| \leq C_2$ for sufficiently large constant $C_1$ and some constant $C_2>0$, we use the Hoeffding bound and then obtain 
    $\pi(A')-1/2\approx (1+\epsilon_1)(\pi(A)-1/2)-O(\sqrt{\log n/n})\geq (1+(\epsilon_1/2))(\pi(A)-1/2)$~w.h.p. Hence, $O(\log n)$ rounds suffice to yield a constant bias.
    (Note that this argument holds when $|\pi(A)-1/2|\leq C_2$ due to the remainder term of Taylor's theorem.)
    
    \item If $C_3\leq \pi(A)<1/2$, 
    it is straightforward to see that $\pi(A')=H_f(\pi(A))+O(\sqrt{\log n/n})\leq \pi(A)-\epsilon_2$ w.h.p.~for some constant $\epsilon_2>0$. Note that we invoke the property that $H_f(x)<x$ whenever $0<x<1/2$.
    
    \item If $\pi(A) \leq C_3$ for sufficiently small constant $C_3$, we use the Markov inequality to show $\pi(A_t)=O(n^{-3})$ w.h.p.~for some $t=O(\log n)$.
    Since $\pi(A)\geq 1/n^2$ whenever $A\neq \emptyset$, this implies that the consensus time is $O(\log n)$ w.h.p.
    Note that, since $H_f'(0)<1$, we have $\E[\pi(A')]\leq H_f(\pi(A))+O(\pi(A)/\sqrt{n})\approx H_f'(0)\pi(A)+O(\pi(A)/\sqrt{n})\leq (1-\epsilon_3)\pi(A)$ for some constant $\epsilon_3>0$.
\end{itemize}

In the proof of \cref{thm:bok}, we modify \cref{lem:nazolemma_informal} and 
apply the same argument.

\section[Estimation of E(pi(A')) and Var(pi(A'))]{Estimation of $\E[\pi(A')]$ and $\Var[\pi(A')]$}
\label{sec:proof_of_lem_gen_square}

In this section, we prove \cref{lem:gen_square} by showing \cref{lem:second_approx,lem:approx_V}, which are generalizations of \cref{lem:gen_square} in terms of \emph{reversible Markov chain}.
This enables us to evaluate $\E[\pi(A')]$ and $\Var[\pi(A')]$ for \fvoting with respect to a $C^2$ function $f$ (see~\cref{sec:estimate_nonsymmetric} for \fvoting with respect to non-symmetric $f$).


\subsection{Technical tools for reversible Markov chains}
To begin with, we briefly summarize the notation of Markov chain, which we will use in this section\footnote{For further detailed arguments about reversible Markov chains, see e.g.,~\cite{LP17}.}. 
Let $V$ be a set of size $n$.
A \emph{transition matrix} $P$ over $V$ is a matrix $P\in [0,1]^{V\times V}$ satisfying $\sum_{v\in V}P(u,v)=1$ for any $u\in V$.
Let $\pi\in [0,1]^V$ denote the {\em stationary distribution} of $P$, i.e., a probability distribution satisfying $\pi P=\pi$.
A transition matrix $P$ is \emph{reversible} if 
$
\pi(u)P(u,v)=\pi(v)P(v,u)
$
for any $u,v\in V$.
It is easy to check that 
the matrix \cref{def:SRW} is
a
reversible transition matrix and its stationary distribution is \cref{def:degreedist}. 
Let $\lambda_1\geq \cdots\geq \lambda_n$ denote the eigenvalues of $P$.
If $P$ is reversible, it is known that $\lambda_i\in \mathbb{R}$ for all~$i$. 
Let $\lambda=\max\{|\lambda_2|, |\lambda_n|\}$ be the second largest eigenvalue in absolute value\footnote{If $P$ is {\em ergodic}, i.e., for any $u,v\in V$, there exists a $t>0$ such that $P^t(u,v)>0$ and $\textrm{GCD}\{t>0: P^t(x,x)>0\}=1$, $1>\lambda_2$ and $\lambda_n>-1$. For example, the transition matrix of the simple random walk on a connected and non-bipartite graph is ergodic.}.

For a function $h:\mathbb{R}\to \mathbb{R}$ and subsets $S,T\subseteq V$, consider the quantity $Q_h(S,T)$ defined~as
\begin{align}
Q_h(S,T)\defeq \sum_{v\in S}\pi(v)h\bigl(P(v,T)\bigr).
\label{def:Qf}
\end{align}
The special case of $h(x)=x$, that is,
$
Q(S,T)\defeq \sum_{v\in S}\pi(v)P(v,T),
$
is well known as {\em edge measure}~\cite{LP17} or {\em ergodic flow}~\cite{AF02,MT06}.
Note that, for any reversible $P$ and subsets $S,T\subseteq V$, $Q(S,T)=Q(T,S)$ holds.
The following result is well known as a version of the {\em expander mixing lemma}.
\begin{lemma}[See, e.g.,~p.163 of \cite{LP17}] \label{lem:EML}
Suppose $P$ is reversible.
Then, for any $S,T\subseteq V$,
\begin{align*}
\left| Q(S,T) - \pi(S)\pi(T) \right| \leq \lambda\sqrt{\pi(S)\pi(T)\bigl(1-\pi(S)\bigr)\bigl(1-\pi(T)\bigr)}.
\end{align*}
\end{lemma}
We show the following lemma which gives a useful estimation of $Q_h(S,T)$.
\begin{lemma}\label{lem:second_approx}
Suppose $P$ is reversible. 
Then, for any $S, T\subseteq V$ and 
any $C^2$ function $h:\mathbb{R}\to \mathbb{R}$, 
\begin{align*}
\Bigl|Q_h(S,T)-\pi(S)h\bigl(\pi(T)\bigr)-h'\bigl(\pi(T)\bigr)\bigl(Q(S,T)-\pi(S)\pi(T)\bigr)\Bigr|
\leq \frac{K_2(h)}{2}\lambda^2\pi(T)\bigl(1-\pi(T)\bigr).
\end{align*}
\end{lemma}

\begin{proof}[Proof of \cref{lem:second_approx}]
From Taylor's theorem, it holds for any $x,y\in [0,1]$ that
\begin{align*}
\left|h(x)-h(y)-h'(y)(x-y)\right|\leq \frac{K_2(h)}{2}(x-y)^2.
\end{align*}
Hence
\begin{align*}
\lefteqn{\Bigl|Q_h(S,T)-\pi(S)h\bigl(\pi(T)\bigr)-h'\bigl(\pi(T)\bigr)\bigl(Q(S,T)-\pi(S)\pi(T)\bigr)\Bigr|} \\
&=\left|\sum_{v\in S}\pi(v)\Bigl(h\bigl(P(v,T)\bigr)-h\bigl(\pi(T)\bigr)-h'\bigl(\pi(T)\bigr)\bigl(P(v,T)-\pi(T)\bigr)\Bigr)\right| \\ 
&\leq \sum_{v\in S}\pi(v)\Bigl|h\bigl(P(v,T)\bigr)-h\bigl(\pi(T)\bigr)-h'\bigl(\pi(T)\bigr)\bigl(P(v,T)-\pi(T)\bigr)\Bigr| \\
&\leq \sum_{v\in S}\pi(v)\frac{K_2(h)}{2}\bigl(P(v,T)-\pi(T)\bigr)^2
\leq \frac{K_2(h)}{2}\sum_{v\in V}\pi(v)\bigl(P(v,T)-\pi(T)\bigr)^2\\
&\leq \frac{K_2(h)}{2}\lambda^2\pi(T)\bigl(1-\pi(T)\bigr).
\end{align*}
Note that the last inequality follows from \cref{lem:second_approx_lambda}.
\end{proof}
%

Next, consider
\begin{align}
R_h(S,T)\defeq \sum_{v\in S}\pi(v)^2h\bigl(P(v,T)\bigr)
\label{def:Rf}
\end{align}
for a function $h:\mathbb{R}\to \mathbb{R}$ and $S,T\subseteq V$.
For notational convenience, for $S\subseteq  V$, let
$
\pi_2(S)\defeq \sum_{v\in S}\pi(v)^2.
$
We show the following lemma that evaluates $R_h(S,T)$.
\begin{lemma}\label{lem:approx_V}
Suppose that $P$ is reversible. 
Then, for any $S, T\subseteq V$ and any $C^2$ function $h:\mathbb{R}\to \mathbb{R}$, 
\begin{align*}
\left|R_h(S,T)-\pi_2(S) h\bigl(\pi(T)\bigr)\right|&\leq K_1(h)\|\pi\|_3^{3/2} \lambda\sqrt{\pi(T)\bigl(1-\pi(T)\bigr)}.
\end{align*}
\end{lemma}
\begin{proof}

We first observe that
\begin{align}
\bigl|h(x)-h(y)\bigr|\leq K_1(h)|x-y|
\label{eq:K_1lip}
\end{align}
holds for any $x,y\in [0,1]$ from Taylor's theorem.
Hence,
\begin{align*}
\lefteqn{\Bigl|R_h(S,T)-\pi_2(S)h\bigl(\pi(T)\bigr)\Bigr|}\\
&=\left|\sum_{v\in S}\pi(v)^2\Bigl(h\bigl(P(v,T)\bigr)-h\bigl(\pi(T)\bigr)\Bigr)\right| 
\leq \sum_{v\in S}\pi(v)^2\Bigl|h\bigl(P(v,T)\bigr)-h\bigl(\pi(T)\bigr)\Bigr|  \\
&\leq \sum_{v\in S}\pi(v)^2K_1(h)\bigl|P(v,T)-\pi(T)\bigr| 
\leq K_1(h)\sum_{v\in V}\pi(v)^2\bigl|P(v,T)-\pi(T)\bigr|. 
\end{align*}
Then, applying the Cauchy-Schwarz inequality and \cref{lem:second_approx_lambda}, 
\begin{align*}
\sum_{v\in V}\pi(v)^2\bigl|P(v,T)-\pi(T)\bigr|
&\leq \sqrt{\left(\sum_{v\in V}\pi(v)^3\right)\left(\sum_{v\in V}\pi(v)\bigl(P(v,T)-\pi(T)\bigr)^2\right)}\\
&\leq \|\pi\|_3^{3/2} \lambda\sqrt{\pi(T)\bigl(1-\pi(T)\bigr)}
\end{align*}
and we obtain the claim.
\end{proof}

\begin{remark}
The results of this paper can be extended to voting processes
where the sampling probability is
determined by a reversible
transition matrix $P$.
This includes voting processes
on edge-weighted graphs $G=(V,E,w)$, where $w:E\to \mathbbm{R}$ denotes an edge weight function.
Consider the transition matrix
$P$ defined as follows: $P(u,v)= w(\{u,v\})/\sum_{x:\{u,x\}\in E}w(\{u,x\})$ for $\{u,v\}\in E$ and $P(u,v)=0$ for $\{u,v\}\notin E$.
A \emph{weighted} \fvoting with respect to $f$
is determined by $\Pr[v\in A'|v\in B]=f(P(v,B))$ and $\Pr[v\in B'|v\in A]=f(P(v,A))$.
For simplicity, in this paper, we do not explore the weighted variant and focus on the usual setting where $P$ is the matrix \cref{def:SRW} and its stationary distribution $\pi$ is \cref{def:degreedist}.
\end{remark}

\subsection[Proof of Lemma 2.4]{Proof of \cref{lem:gen_square}}
For the first inequality, by substituting $V$ to $S$ of \cref{lem:second_approx}, we obtain
\begin{align*}
    \Bigl|Q_h(V,T)-h\bigl(\pi(T)\bigr)\Bigr|
& \leq \frac{K_2(h)}{2}\lambda^2\pi(T)\bigl(1-\pi(T)\bigr).
\end{align*}
Note that $Q(V,T)=Q(T,V)=\pi(T)$ from the reversibility of $P$.
Similarly, we obtain the second inequality by substituting $V$ to $S$ of \cref{lem:approx_V}.\hspace{\fill}\qedsymbol

%

\subsection[Non-symmetric functions]{Non-symmetric functions}
\label{sec:estimate_nonsymmetric}
This section is devoted to 
evaluate $\E[\pi(A')]$ and $\Var[\pi(A')]$ 
for non-symmetric $f$.
To be more specifically, we prove the following.
\begin{lemma} \label{lem:expectation}
Consider a \fvoting with respect to a $C^2$ function $f$ on a $\lambda$-expander graph.
Then, 
for all $A\subseteq V$, 
\begin{align*}
\left|\E[\pi(A')]-H_f\bigl(\pi(A)\bigr)\right|
&\leq K_2(f)\lambda \bigl(|2\pi(A)-1|+\lambda\bigr)\pi(A)\bigl(1-\pi(A)\bigr). 
\end{align*}
\end{lemma}
\begin{lemma} \label{lem:variance}
Consider a \fvoting with respect to a $C^2$ function $f$ on a $\lambda$-expander graph.
Let $g(x)\defeq f(x)(1-f(x))$.
Then,
for all $A\subseteq V$,
\begin{align*}
\left|\Var[\pi(A')]-\|\pi\|_2^2g\left(\frac{1}{2}\right)\right|&\leq K_1(g)\left(\frac{1}{2}\|\pi\|_2^2\left|2\pi(A)-1\right|+2\|\pi\|_3^{3/2}\lambda \sqrt{\pi(A)\bigl(1-\pi(A)\bigr)}\right).
\end{align*}
\end{lemma}
Recall that 
we use $B=V\setminus A$ for $A\subseteq V$.
Then, it is clear that
\begin{align}
    &\E[\pi(A')]
    =\pi(A)-\sum_{v\in A}\pi(v)f\bigl(P(v,B)\bigr)+\sum_{v\in B}\pi(v)f\bigl(P(v,A)\bigr) 
    \label{eq:EpiAtoQf}, \\
    &\Var[\pi(A')]
    =\sum_{v\in A}\pi(v)^2g\bigl(P(v,B)\bigr)+\sum_{v\in B}\pi(v)g\bigl(P(v,A)\bigr).
    \label{eq:VpiAtoRg}
\end{align}

\begin{proof}[Proof of \cref{lem:expectation}]
From \cref{def:fVoting}, \cref{eq:EpiAtoQf,def:Qf}, we have
\begin{align}
\E[\pi(A')]
&=\pi(A)-Q_f(A,B)+Q_f(B,A), \label{eq:expectedpiA} \\ 
H_f\bigl(\pi(A)\bigr)
&=\pi(A)-\pi(A)f\bigl(\pi(B)\bigr)+\pi(B)f\bigl(\pi(A)\bigr). \label{eq:HfpiA}
\end{align}
For notational convenience, for $S,T\subseteq V$, let 
\begin{align*}
\Delta_f(S,T)
&\defeq Q_f(S,T)-\pi(S)f\bigl(\pi(T)\bigr)-f'\bigl(\pi(T)\bigr)\bigl(Q(S,T)-\pi(S)\pi(T)\bigr)\\
&=Q_h(S,T)-\pi(S)f\bigl(\pi(T)\bigr)-f'\bigl(\pi(T)\bigr)\bigl(Q(T,S)-\pi(T)\pi(S)\bigr).
\end{align*}
The equality follows from the reversibility of $P$ (see~\cref{sec:proof_of_lem_gen_square}).
From \cref{lem:second_approx}, we have
\begin{align*}
    |\Delta_f(S,T)|\leq \frac{K_2(f)}{2}\lambda^2\pi(T)\bigl(1-\pi(T)\bigr).
\end{align*}
Then, combining \cref{eq:expectedpiA,eq:HfpiA}, we have
\begin{align*}
\lefteqn{\Bigl|\E[\pi(A')]-H_f\bigl(\pi(A)\bigr)\Bigr|}\\
&=\Bigl|
Q_f(B,A)-\pi(B)f\bigl(\pi(A)\bigr)
-Q_f(A,B)+\pi(A)f\bigl(\pi(B)\bigr)
\Bigr|\\
&=\Bigl|\Delta_f(B,A)+f'\bigl(\pi(A)\bigr)\bigl(Q(A,B)-\pi(A)\pi(B)\bigr)\Bigr. \\
&\hspace{2em} \Bigl.-\Delta_f(A,B)-f'\bigl(\pi(B)\bigr)\bigl(Q(A,B)-\pi(A)\pi(B)\bigr)\Bigr|\\
&\leq 
|\Delta_f(B,A)|+|\Delta_f(A,B)|+
\Bigl|f'\bigl(\pi(A)\bigr)-f'\bigl(\pi(B)\bigr)\Bigr|\bigl|Q(A,B)-\pi(A)\pi(B)\bigr|\\
&\leq  K_2(f)\lambda^2\pi(A)\pi(B)+K_2(f)\bigl|\pi(A)-\pi(B)\bigr|\lambda\pi(A)\pi(B).
\end{align*}
and we obtain the claim.
Note that the last inequality follows from Taylor's theorem~\cref{eq:K_1lip} and \cref{lem:EML}.
\end{proof}

\begin{proof}[Proof of \cref{lem:variance}]
From \cref{eq:VpiAtoRg,def:Rf}, 
\begin{align*}
\Var[\pi(A')]
=R_g(A,B)+R_g(B,A).
\end{align*}
Thus, applying \cref{lem:approx_V} yields
\begin{align}
\left|\Var\bigl[\pi(A')\bigr]-\Bigl(\pi_2(A)g\bigl(\pi(B)\bigr)+\pi_2(B)g\bigl(\pi(A)\bigr)\Bigr)\right|\leq 2K_1(g)\|\pi\|_3^{3/2}\lambda\sqrt{\pi(A)\pi(B)}.
\label{eq:variance_piA_gen}
\end{align}
Next, using Taylor's theorem \cref{eq:K_1lip},
\begin{align}
\lefteqn{\left|\pi_2(A)g\bigl(\pi(B)\bigr)+\pi_2(B)g\bigl(\pi(A)\bigr)-\|\pi\|_2^2g\left(\frac{1}{2}\right)\right|}\nonumber \\
&=\left| \pi_2(A)\left(g\bigl(\pi(B)\bigr)-g\left(\frac{1}{2}\right) \right)
+\pi_2(B)\left(g\bigl(\pi(A)\bigr)-g\left(\frac{1}{2}\right) \right)\right|\nonumber \\
&\leq K_1(g)\pi_2(A)\left|\pi(B)-\frac{1}{2}\right|+K_1(g)\pi_2(B)\left|\pi(A)-\frac{1}{2}\right|
=K_1(g)\|\pi\|_2^2\left|\pi(A)-\frac{1}{2}\right|. \label{eq:VarpiA_prof_2}
\end{align}
The last equality follows since $|\pi(A)-1/2|=|\pi(B)-1/2|$.
Combining \cref{eq:variance_piA_gen,eq:VarpiA_prof_2}, we obtain the claim.
\end{proof}

\section[Proofs of Theorem 1.3 and 1.5]{Proofs of \cref{thm:mainthm,thm:fastconsensus}}
\label{sec:dynamicsofpia}
Consider a \qmaj \fvoting with respect to $f$ on an $n$-vertex $\lambda$-expander graph with degree distribution $\pi$.
Let $A_0,A_1,\ldots,$ be the sequence given by the \fvoting with initial configuration $A_0\subseteq V$.
\Cref{thm:mainthm,thm:fastconsensus} follow from the following lemma.
\begin{lemma}
\label{lem:mainlem}
Consider a \qmaj \fvoting with respect to $f$ on an $n$-vertex $\lambda$-expander graph with degree distribution $\pi$.
Let 
$\epsilon_h(f)\defeq H_f'(1/2)-1$, 
$\epsilon_c(f)\defeq 1-H_f'(0)$ and 
$K(f)\defeq \max\{K_2(f),K_2(H_f)\}$ be three positive constants depending only on $f$.
Then, the following holds:
\begin{enumerate}[label=(\Roman*)]
    \item \label{lab:phase1}
    Let $C_1>0$ be an arbitrary constant and $\varepsilon:\mathbb{N}\to\mathbb{R}$ be an arbitrary function satisfying $\varepsilon(n)\to 0$ as $n\to \infty$.
    Suppose that $\lambda \leq C_1n^{-1/4}$, $\|\pi\|_2 \leq C_1/\sqrt{n}$ and $\|\pi\|_3\leq \varepsilon/\sqrt{n}$.
    Then, for any $A_0\subseteq V$ such that $|\delta(A_0)|\leq c_1 \log n/\sqrt{n}$ for an arbitrary constant $c_1>0$, $|\delta(A_t)|\geq c_1\log n/\sqrt{n}$ within $t=O(\log n)$ steps w.h.p.
    \item \label{lab:phase2}
    Suppose that $\lambda\leq \frac{\epsilon_h(f)}{2K(f)}$.
    Then, for any $A_0\subseteq V$ s.t.~$\frac{2\max\{K(f),8\}}{\epsilon_h(f)}\max\{\lambda^2,\|\pi\|_2\sqrt{\log n}\}\\ \leq |\delta(A_0)|\leq \frac{\epsilon_h(f)}{K(f)}$, $|\delta(A_t)|\geq \frac{\epsilon_h(f)}{K(f)}$ within $t=O(\log |\delta(A_0)|^{-1})$ steps w.h.p.
    \item \label{lab:phase3}
    Let $c_2, c_3$ be two arbitrary constants satisfying $0<c_2<c_3<1/2$ and $\epsilon(f)\defeq \min_{x\in [c_2, c_3]}\bigl(x-H_f(x)\bigr)$ be a positive constant depending $f, c_2, c_3$.
    Suppose that $\lambda\leq \frac{\epsilon(f)}{2K(f)}$ and $\|\pi\|_2\leq \frac{\epsilon(f)}{4\sqrt{\log n}}$.
    Then, for any $A_0\subseteq V$ satisfying $c_2\leq \pi(A_0)\leq c_3$, $\pi(A_t)\leq c_2$ within constant steps w.h.p.
    \item \label{lab:phase4}
    Suppose that $\lambda \leq \frac{\epsilon_c(f)}{2K(f)}$ and $\|\pi\|_2\leq
    \frac{\epsilon_c(f)^2}{32K(f)\sqrt{\log n}}
    $.
    Then, for any $A_0\subseteq V$ satisfying $\pi(A_0)\leq \frac{\epsilon_c(f)}{8K(f)}$, $\pi(A_t)=0$ within $t=O(\log n)$ steps w.h.p.
    \item \label{lab:phase5}
    Suppose that $H_f'(0)=0$, $\lambda \leq \frac{1}{10K(f)}$ and $\|\pi\|_2\leq \frac{1}{64K(f)\sqrt{\log n}}$.
    Then, for any $A_0\subseteq V$ satisfying $\pi(A_0)\leq \frac{1}{7K(f)}$, it holds w.h.p.~that $\pi(A_t)=0$ within $$t=O\left(\log \log n + \frac{\log n}{\log \lambda^{-1}} + \frac{\log n}{\log (\|\pi\|_2\sqrt{\log n})^{-1}}\right)\,  \textrm{steps}.$$ 
\end{enumerate}
\end{lemma}

\begin{proof}[Proof of \cref{thm:mainthm}\cref{thm:cons_initial_bias}]
Since $\|\pi\|_2\geq 1/\sqrt{n}$, we have $|\delta(A_0)|=\Omega(\sqrt{\log n/n})$.
This implies that Phase \cref{lab:phase2} takes at most $O(\log n)$.
Thus, we obtain the claim since we can merge Phases \cref{lab:phase2} to \cref{lab:phase4} by taking appropriate constants $c_2, c_3$ in Phase \cref{lab:phase3}.
\end{proof}
\begin{proof}[Proof of \cref{thm:mainthm}\cref{thm:worstcasefvoting}]
Under the assumption of \cref{thm:mainthm}\cref{thm:worstcasefvoting}, 
for any positive constant $C$, a positive constant $C'$ exists such that $C(\lambda^2+\|\pi\|_2\sqrt{\log n})\leq  C'\frac{\log n}{\sqrt{n}}$.
Thus, we can combine Phase~\cref{lab:phase1} and \cref{thm:mainthm}\cref{thm:cons_initial_bias}, and we obtain the claim.
\end{proof}
\begin{proof}[Proof of \cref{thm:fastconsensus}]
Combining Phases \cref{lab:phase2}, \cref{lab:phase3} and \cref{lab:phase5}, we obtain the claim.
\end{proof}
\subsection[Proof of Lemma 4.1]{Proof of \cref{lem:mainlem}}
\label{sec:ommited_proofs}
For notational convenience, let 
\begin{align*}
&\alpha\defeq \pi(A),\, \alpha'\defeq \pi(A'),\,\alpha_t\defeq \pi(A_t), \\
&\delta\defeq \delta(A)=2\alpha-1,\, \delta'\defeq \delta(A'),\, \delta_t\defeq\delta(A_t).
\end{align*}

\subsection[Phase 1]{Phase \cref{lab:phase1}:~$0\leq |\delta|\leq c_1\log n/\sqrt{n}$}
\label{sec:symmetrybreaking}
We use the following lemma to show \cref{lem:mainlem}\cref{lab:phase1}.
\begin{lemma}[Lemma 4.5 of \cite{CGGNPS18}] \label{lem:nazolemma}
Consider a Markov chain $(X_t)_{t=1}^\infty$ with finite state space $\Omega$ and a function $\Psi:\Omega \to \{0,\ldots,n\}$.
Let $C_3$ be arbitrary constant and $m=C_3 \sqrt{n}\log n$.
Suppose that $\Omega,\Psi$ and $m$ satisfies the following conditions:
\begin{enumerate}[label=$(\roman*)$]
\item\label{state:nazo_sqrt} For any positive constant $h$, there exists a positive constant $C_1<1$ such that
\begin{align*}
\Pr\left[\Psi(X_{t+1})<h\sqrt{n}\,\middle|\,\Psi(X_t)\leq m\right] < C_1.
\end{align*}
\item\label{state:nazo_cher} Three positive constants $\gamma, C_2$ and $h$ exist such that, for any $x\in \Omega$ satisfying $h\sqrt{n}\leq \Psi(x)<m$,
\begin{align*}
\Pr\left[\Psi(X_{t+1})< (1+\gamma)\Psi(X_t) \,\middle|\, X_t=x\right] < \exp\left(-C_2\frac{\Psi(x)^2}{n}\right).
\end{align*}
\end{enumerate}
Then, $\Psi(X_t) \geq m$ holds w.h.p.~for some $t=O(\log n)$.
\end{lemma}
Let us first prove the following lemma concerning the growth rate of $|\delta|$, which we will use in the proofs of \cref{lab:phase1} and \cref{lab:phase2} of \cref{lem:mainlem}.
\begin{lemma}
\label{lem:delta_lower_half}
Consider a \qmaj \fvoting with respect to $f$ on an $n$-vertex $\lambda$-expander graph with degree distribution $\pi$.
Let 
$\epsilon_h(f)\defeq H_f'(1/2)-1$ and 
$K(f)\defeq \max\{K_2(f),K_2(H_f)\}$ be positive constants depending only on $f$.
Suppose that $\lambda \leq \frac{\epsilon_h(f)}{2K(f)}$.
Then, for any $A\subseteq V$ satisfying $\frac{2K(f)}{\epsilon_h(f)}\lambda^2 \leq |\delta|\leq \frac{\epsilon_h(f)}{K(f)}$,
\begin{align*}
\Pr\left[|\delta'|\leq \left(1+\frac{\epsilon_h(f)}{8}\right)|\delta|\right]
&\leq 2\exp\left(-\frac{\epsilon_h(f)^2\delta^2}{128\|\pi\|_2^2}\right).
\end{align*}
\end{lemma}
\begin{proof}
Combining \cref{lem:expectation} and Taylor's theorem, we have
\begin{align}
\left|\E[\delta']-H_f'\left(\frac{1}{2}\right)\delta\right|
&=2\left|\E[\alpha']-\frac{1}{2}-H_f'\left(\frac{1}{2}\right)\left(\alpha-\frac{1}{2}\right)\right|\nonumber \\
&=2\left|\E\left[\alpha'\right]-H_f\left(\alpha\right)+H_f\left(\alpha \right)-H_f\left(\frac{1}{2}\right)-H_f'\left(\frac{1}{2}\right)\left(\alpha-\frac{1}{2}\right)\right|\nonumber \\
&\leq 2K_2(f)\lambda\left(|\delta|+\lambda\right)\alpha(1-\alpha)+K_2(H_f)\left(\alpha-\frac{1}{2}\right)^2\nonumber \\
&\leq \left(\frac{K(f)}{2}\lambda+\frac{K(f)}{4}|\delta|\right)|\delta|+\frac{K(f)}{2}\lambda^2
\label{eq:delta_approx_halfeq}
\end{align}
Note that $H_f(1/2)=1/2$ from the definition.
From assumptions of $\lambda\leq \frac{\epsilon_h(f)}{2K(f)}$, $|\delta|\leq \frac{\epsilon_h(f)}{K(f)}$ and $\lambda^2\leq \frac{\epsilon_h(f)}{2K(f)}|\delta|$, we have
\begin{align*}
\left|H_f'\left(\frac{1}{2}\right)\delta\right|-\left|\E[\delta']\right|
\leq \left|H_f'\left(\frac{1}{2}\right)\delta-\E[\delta']\right|
\leq \frac{3}{4}\epsilon_h(f)|\delta|.
\end{align*}
Hence, it holds that
\begin{align*}
\bigl|\E[\delta']\bigr|
&\geq \left|H_f'\left(\frac{1}{2}\right)\delta\right|-\frac{3}{4}\epsilon_h(f)|\delta|
= (1+\epsilon_h(f))|\delta|-\frac{3}{4}\epsilon_h(f)|\delta|
= \left(1+ \frac{\epsilon_h(f)}{4}  \right)|\delta|.
\label{eq:delta_Hoeffding}
\end{align*}
We observe that, for any $\kappa>0$, 
\begin{align}
\Pr\left[|\delta'|\leq \bigl|\E[\delta']\bigr| -\kappa\right]\leq 2\exp\left(-\frac{\kappa^2}{2\|\pi\|_2^2}\right)
\end{align}
from \cref{cor:Hoeffding_abs}.
Note that $\delta'=\sum_{v\in V}\pi(v)(2X_v-1)$ for independent indicator random variables $(X_v)_{v\in V}$ (see~\cref{eq:Xvdef} for the definition of $X_v$).
Thus,
\begin{align*}
\Pr\left[|\delta'|\leq \left(1+\frac{\epsilon_h(f)}{8}\right)|\delta|\right]
&=\Pr\left[|\delta'|\leq \left(1+\frac{\epsilon_h(f)}{4}\right)|\delta|-\frac{\epsilon_h(f)}{8}|\delta|\right]\\
&\leq \Pr\left[|\delta'|\leq \bigl|\E[\delta']\bigr|-\frac{\epsilon_h(f)}{8}|\delta|\right]
\leq 2\exp\left(-\frac{\epsilon_h(f)^2\delta^2}{128\|\pi\|_2^2}\right)
\end{align*}
and we obtain the claim. 
\end{proof}

\begin{proof}[Proof of \cref{lem:mainlem}\cref{lab:phase1}]
We check the conditions \cref{state:nazo_sqrt} and \cref{state:nazo_cher} of \cref{lem:nazolemma} with letting $\Psi(A)= \lfloor n |\delta(A)|\rfloor$ and $m=c_1\sqrt{n}\log n$.

\paragraph*{Condition \cref{state:nazo_sqrt}.}
First, we show the following claim 
that evaluates $\Var[\delta']$.

\begin{claim}
\label{cor:variancebound}
Under the same assumption as \cref{lem:mainlem}\cref{lab:phase1}, 
\begin{align*}
    \frac{\epsilon_{{\rm var}}(f)}{n} \leq \Var[\delta']\leq \frac{5C_1^2}{n}
\end{align*}
where $\epsilon_{{\rm var}}(f)\defeq f(1/2)(1-f(1/2))$ is a positive constant depending only on $f$.
\end{claim}
\begin{proof}[Proof of the claim]
From \cref{lem:variance} and assumptions, we have
\begin{align*}
\left|\frac{\Var[\delta']}{4}-\|\pi\|_2^2g\left(\frac{1}{2}\right)\right|
&=\left|\Var[\alpha']-\|\pi\|_2^2g\left(\frac{1}{2}\right)\right|
\leq K_1(g)\left(\|\pi\|_2^2\frac{|\delta|}{2}+\|\pi\|_3^{3/2}\lambda\right)\\
&\leq \frac{K_1(g)}{n}\left(C_1^2c_1\frac{\log n}{\sqrt{n}}+C_1\epsilon^{3/2}\right)
=\frac{1}{n}\cdotp o(1).
\end{align*}
Note that $\Var[\delta']=\Var[2\alpha'-1]=4\Var[\alpha']$.
Since $\|\pi\|_2^2\geq 1/n$, we have
\begin{align*}
\frac{\epsilon_{{\rm var}}(f)}{n}\leq \frac{4\epsilon_{{\rm var}}(f)-o(1)}{n}\leq \Var[\delta']\leq \frac{4C_1^2+o(1)}{n}\leq \frac{5C_1^2}{n}.
\end{align*}
\end{proof}

From \cref{cor:BEbound_cor2} with letting $Y_v=\pi(v)(2X_v-1)$, we have
\begin{align}
\Pr\left[\left|\delta'\right|\leq x\sqrt{\frac{\epsilon_{{\rm var}}(f)}{n}}\right]
&\leq \Pr\left[\left|\delta'\right|\leq x\sqrt{\Var[\delta']}\right] \leq \Phi(x)+\frac{5.6\|\pi\|_3^3}{\Var[\delta']^{3/2}}\nonumber \\
&\leq \Phi(x)+5.6\frac{\epsilon^3}{n^{3/2}}\cdotp \frac{n^{3/2}}{\epsilon_{{\rm var}}(f)^{3/2}}
= \Phi(x)+o(1)
\label{eq:biasprob}
\end{align}
for any $x\in \mathbb{R}$, where $\Phi(x)=\frac{1}{\sqrt{2\pi}}\int_{-\infty}^x\mathrm{e}^{-y^2/2}\mathrm{d}y$.
Thus, for any constant $h>0$, there exists some constant $C>0$ such that
\begin{align*}
\Pr[\Psi(A')<h\sqrt{n}\mid \Psi(A)\leq m] < C, 
\end{align*}
which verifies the condition \cref{state:nazo_sqrt}.
\paragraph*{Condition \cref{state:nazo_cher}.}
Set $h=\frac{2K(f)}{\epsilon_h(f)} C_1^2$ and assume $h\sqrt{n}\leq \Psi(A)<m$. Then
\begin{align*}
\frac{2K(f)}{\epsilon_h(f)} \lambda^2 n \leq \frac{2K(f)}{\epsilon_h(f)} C_1^2\sqrt{n} = h\sqrt{n} \leq \Psi(A) \leq |\delta|n=o(n).
\end{align*}
Thus, we can apply \cref{lem:delta_lower_half}
and positive constants $\gamma, C$ exist
such that, for any $h\sqrt{n} \leq \Psi(A)\leq c_1\sqrt{n}\log n$,
\begin{align*}
\Pr[\Psi(A')<(1+\gamma)\Psi(A)]
&< \exp\left(-C\frac{\Psi(A)^2}{n}\right).
\end{align*}
Note that $\|\pi\|_2^2=\Theta(1/n)$ from the assumption.
This verifies the condition \cref{state:nazo_cher}.

Thus, we can apply \cref{lem:nazolemma} and we obtain the claim.
\end{proof}

\subsection[Phase 2]{Phase \cref{lab:phase2}: $\frac{2\max\{K(f),8\}}{\epsilon_h(f)}\max\{\lambda^2,\|\pi\|_2\sqrt{\log n}\}\leq |\delta|\leq \frac{\epsilon_h(f)}{K(f)}$}
\begin{proof}[Proof of \cref{lem:mainlem}\cref{lab:phase2}]
Since $|\delta|\geq \frac{16}{\epsilon_h(f)}\|\pi\|_2\sqrt{\log n}$ from assumptions, applying \cref{lem:delta_lower_half} yields
\begin{align*}
\Pr\left[|\delta'|\leq \left(1+\frac{\epsilon_h(f)}{8}\right)|\delta|\right]\leq \frac{2}{n^2}.
\end{align*}
Thus, it holds with probability larger than $(1-2/n^2)^t$ that
$
|\delta_t|\geq \left(1+\frac{\epsilon_h(f)}{8}\right)^t|\delta_0|
$
and we obtain the claim by substituting $t=O(\log |\delta_0|^{-1})$.
\end{proof}
\subsection[Phase 3]{Phase \cref{lab:phase3}: $0<c_2\leq \alpha \leq c_3<1/2$}
\begin{proof}[Proof of \cref{lem:mainlem}\cref{lab:phase3}]
We first observe that, for any $\kappa>0$, 
\begin{align}
    \Pr\left[\left|\alpha'-\E[\alpha']\right|\geq \kappa\|\pi\|_2\sqrt{\log n}\right]&\leq 2n^{-2\kappa}
\label{eq:hoeffding_alpha}
\end{align}
from \cref{lem:Hoeffding}.
Note that $\alpha'=\sum_{v\in V}\pi(v)X_v$ for independent indicator random variables $(X_v)_{v\in V}$.
Hence, applying \cref{lem:expectation} yields
\begin{align}
\left|\alpha'-H_f(\alpha)\right|
\leq \left|\alpha'-\E[\alpha']\right|+\left|\E[\alpha']-H_f(\alpha)\right|
\leq \|\pi\|_2\sqrt{\log n}+\frac{K_2(f)}{4}(|\delta|+\lambda)\lambda
\label{eq:pia'bound}
\end{align}
with probability larger than $1-2/n^2$.
Then, for any $\alpha\in [c_2, c_3]$, it holds with probability larger than $1-2/n^2$ that 
\begin{align*}
\alpha'
&\leq H_f(\alpha)+\frac{K(f)}{2}\lambda+\|\pi\|_2\sqrt{\log n}
\leq \alpha-\epsilon(f)+\frac{\epsilon(f)}{4}+\frac{\epsilon(f)}{4}
\leq \alpha-\frac{\epsilon(f)}{2}.
\end{align*}
Thus, for $\alpha_0\in [c_2, c_3]$, $\alpha_t\leq c_2$ within $t=2(c_3-c_2)/\epsilon(f)=O(1)$ steps w.h.p.
\end{proof}

\subsection[Phase 4]{Phase \cref{lab:phase4}: $0\leq \alpha \leq \frac{\epsilon_c(f)}{8K(f)}$}
We show the following lemma which is useful for proving \cref{lab:phase4} and \cref{lab:phase5} of \cref{lem:mainlem}.
\begin{lemma}
\label{lem:consensus_lessthanone}
Let $\epsilon\in (0,1]$ be an arbitrary constant.
Consider \fvoting on an $n$-vertex connected graph with degree distribution $\pi$.
Suppose that, for some $\alpha_*\in [0,1]$ and $K\in [0, 1-\epsilon]$, 
$$\E[\alpha']\leq K\alpha$$
for any $A\subseteq V$ satisfying $\alpha\leq \alpha_*$ and $\|\pi\|_2\leq \frac{\epsilon\alpha_*}{2\sqrt{\log n}}$.
Then, for any $A_0\subseteq V$ satisfying $\alpha_0\leq \alpha_*$, $\alpha_t=0$ w.h.p.~within $O\left(\frac{\log n}{\log K^{-1}}\right)$ steps.
\end{lemma}
\begin{proof}
For any $\alpha\leq \alpha_*$, from \cref{eq:hoeffding_alpha} and assumptions of $\E[\alpha']\leq \alpha$ and $\|\pi\|_2\leq \frac{\epsilon\alpha_*}{2\sqrt{\log n}}$, it holds with probability larger than $1-2/n^4$ that
\begin{align*}
\alpha'
\leq \E[\alpha']+2\|\pi\|_2\sqrt{\log n}
\leq K\alpha+\epsilon \alpha_*
\leq (1-\epsilon)\alpha_*+\epsilon \alpha_*
= \alpha_*.
\end{align*}
Thus, for any $\alpha_0\leq \alpha_*$, we have
\begin{align*}
\E[\alpha_t]
&=\sum_{x\leq a_*}\E\left[\alpha_t|\alpha_{t-1}=x\right]\Pr\left[\alpha_{t-1}=x\right]+\sum_{x> a_*}\E\left[\alpha_t|\alpha_{t-1}=x\right]\Pr\left[\alpha_{t-1}=x\right]\\
&\leq \sum_{x\leq a_*}K x\Pr\left[\alpha_{t-1}=x\right]+\Pr\left[\alpha_{t-1}>a_*\right]
\leq K \E[\alpha_{t-1}]+\frac{2t}{n^{4}}\\
&\leq \cdots \leq K^t\alpha_0+\frac{2t^2}{n^4}
\leq K^t+\frac{2t^2}{n^4}.
\end{align*}
This implies that, 
$
\E[\alpha_t]
=O(n^{-3})
$ within $t=O\left(\frac{\log n}{\log K^{-1}}\right)$ steps.
Let $\pi_{\min}\defeq \min_{v\in V}\pi(v)\geq 1/(2|E|)\geq 1/n^2$.
Markov inequality yields
\begin{align*}
\Pr[\alpha_t=0]=1-\Pr[\alpha_t\geq \pi_{\min}]\geq 1-\frac{\E[\alpha_t]}{\pi_{\min}}=1-O(1/n)
\end{align*}
and we obtain the claim. 
\end{proof}

\begin{proof}[Proof of \cref{lem:mainlem} of \cref{lab:phase4}]
Combining \cref{lem:expectation} and Taylor's theorem, 
\begin{align}
\bigl|\E[\alpha']-H_f'(0)\alpha\bigr|
&=\bigl|\E[\alpha']-H_f(\alpha)+H_f(\alpha)-H_f(0)-H_f'(0)(\alpha-0)\bigr|\nonumber \\
&\leq K_2(f)\lambda\left(|\delta|+\lambda\right)\alpha(1-\alpha)+\frac{K_2(H_f)}{2}\alpha^2 \nonumber \\
&\leq 2K(f)\lambda\alpha+\frac{K(f)}{2}\alpha^2. \label{eq:alphaclosetozero}
\end{align}
Hence, for any $\alpha\leq \frac{\epsilon_c(f)}{8K(f)}$, we have
\begin{align*}
\E[\alpha']
&\leq \left(H_f'(0)+ 2K(f)\lambda+\frac{K(f)}{2}\alpha\right)\alpha \\
&\leq \left(1-\epsilon_c(f)+\frac{\epsilon_c(f)}{4}+\frac{\epsilon_c(f)}{4}\right)\alpha
=\left(1-\frac{\epsilon_c(f)}{2}\right)\alpha. 
\end{align*}
Letting $\epsilon=\epsilon_c(f)/2$, $K=1-\epsilon_c(f)/2$ and $\alpha_*=\frac{\epsilon_c(f)}{8K(f)}$,
from the assumption, $\|\pi\|_2\leq \frac{\epsilon_c(f)^2}{32K(f)\sqrt{\log n}}=\frac{\epsilon \alpha_*}{2\sqrt{\log n}}$.
Thus, we can apply \cref{lem:consensus_lessthanone} and we obtain the claim.
\end{proof}

\subsection[Phase 5]{Phase \cref{lab:phase5}: $H_f'(0)=0$ and $0\leq \alpha \leq \frac{1}{7K(f)}$}
\begin{proof}[Proof of \cref{lem:mainlem}\cref{lab:phase5}]
In this case, from \cref{eq:alphaclosetozero},
\begin{align}
    \E[\alpha']\leq 2K(f)\lambda \alpha + \frac{K(f)}{2}\alpha^2.
    \label{eq:alphahfzero}
\end{align}
We consider the following two cases.

\noindent \textbf{Case 1. $\max\left\{\lambda, \sqrt{\frac{\|\pi\|_2\sqrt{\log n}}{K(f)}}\right\}\leq \alpha\leq \frac{1}{7K(f)}$:}
In this case, combining \cref{eq:hoeffding_alpha,eq:alphahfzero}, it holds with probability larger than $1-2/n^2$ that
\begin{align*}
    \alpha'
    &\leq 
    \left(\frac{2K(f)\lambda}{\alpha}+\frac{K(f)}{2}+\frac{\|\pi\|_2\sqrt{\log n}}{\alpha^2}\right)\alpha^2
    \leq \frac{7K(f)}{2}\alpha^2.
\end{align*}
Applying this inequality iteratively, for any $\alpha_0\leq 7K(f)^{-1}$, 
\begin{align*}
    \alpha_t
    \leq \frac{7K(f)}{2}\alpha_{t-1}^2
    \leq \cdots \leq \frac{2}{7K(f)}\left(\frac{7K(f)}{2}\alpha_0\right)^{2^t}
    \leq \frac{2}{7K(f)2^{2^t}}.
\end{align*}
holds with probability larger than $(1-2/n^2)^t$.
This implies that, within $t=O(\log \log n)$ steps, $\alpha_t\leq \max\left\{\lambda, \sqrt{\frac{\|\pi\|_2\sqrt{\log n}}{K(f)}}\right\}$ w.h.p. Note that $\max\left\{\lambda, \sqrt{\frac{\|\pi\|_2\sqrt{\log n}}{K(f)}}\right\}\geq \sqrt{\frac{\|\pi\|_2\sqrt{\log n}}{K(f)}}\geq \sqrt{\frac{\sqrt{\log n/n}}{K(f)}}$ since $\|\pi\|_2^2\geq 1/n$.

\noindent \textbf{Case 2. $\alpha\leq \max\left\{\lambda, \sqrt{\frac{\|\pi\|_2\sqrt{\log n}}{K(f)}}\right\}$:}
Set $\alpha_*=\max\left\{\lambda, \sqrt{\frac{\|\pi\|_2\sqrt{\log n}}{K(f)}}\right\}\geq \sqrt{\frac{\|\pi\|_2\sqrt{\log n}}{K(f)}}$, $K=\frac{5K(f)}{2}\lambda + \frac{1}{2}\sqrt{K(f)\|\pi\|_2\sqrt{\log n}}$ and $\epsilon=1/4$. 
Then, from $\lambda\leq \frac{1}{10K(f)}$ and $\|\pi\|_2\leq \frac{1}{64K(f)\sqrt{\log n}}$, 
\begin{align*}
\|\pi\|_2
&=(\sqrt{\|\pi\|_2})^2\leq \frac{\sqrt{\|\pi\|_2}}{8\sqrt{K(f)\sqrt{\log n}}}
= \sqrt{\frac{\|\pi\|_2\sqrt{\log n}}{K(f)}}\frac{\epsilon}{2\sqrt{\log n}}
\leq \frac{\epsilon\alpha_*}{2\sqrt{\log n}}, \\
    K&\leq \frac{1}{2}+\frac{1}{16}\leq 1-\epsilon, \\
    \E[\alpha']
    &\leq \left(2K(f)\lambda  + \frac{K(f)}{2}\alpha\right)\alpha
    \leq \left(2K(f)\lambda + \frac{K(f)}{2}\lambda + \frac{1}{2}\sqrt{K(f)\|\pi\|_2\sqrt{\log n}}\right)\alpha
    =K\alpha.
\end{align*}
$
    K\leq 1/4+1/2=3/4
$.
Thus, applying \cref{lem:consensus_lessthanone}, we obtain the claim.
\end{proof}

\section[proof of Theorem 1.4]{Proof of \cref{thm:lowerboundconsensustime}}
\label{app:lower}
This section is devoted to prove \cref{thm:lowerboundconsensustime}.
In particular, we show the following theorem.

\begin{theorem}
\label{thm:newlower}
Let $C>0$ be an arbitrary constant.
Consider a \qmaj \fvoting with respect to $f$ on an $n$-vertex $\lambda$-expander graph with degree distribution $\pi$.
Suppose that $\max\{\lambda,\|\pi\|_2\}\leq n^{-C}$.
Then, for any $A\subseteq V$ satisfying $|\delta(A)|\leq n^{-C}$, $T_{\cons}(A)=\Omega(\log n)$ w.h.p.
\end{theorem}
\begin{proof}[Proof of \cref{thm:newlower}]
From \cref{eq:delta_approx_halfeq}, 
\begin{align*}
\left|\E[\delta']\right|
&\leq H_f'\left(\frac{1}{2}\right)\left|\delta\right|+\left(\frac{K(f)}{2}\lambda+\frac{K(f)}{4}|\delta|\right)|\delta|+\frac{K(f)}{2}\lambda^2\\
&\leq \left(1+\epsilon_h(f)+\frac{3K(f)}{4}\right)|\delta|+K(f)\lambda^2.
\end{align*}
Recall that $\delta'=\sum_{v\in V}(2\pi_v-1)$ for independent indicator random variables $(X_v)_{v\in V}$ \cref{eq:Xvdef}.
Thus, for any $\kappa>0$,
\begin{align*}
    \Pr\left[\left|\delta'\right|\geq \left|\E[\delta']\right|+\kappa\right]\leq \exp\left(-\frac{\kappa^2}{2\|\pi\|_2^2}\right)
\end{align*}
from \cref{cor:Hoeffding_abs}.
Hence, it holds with probability larger than $1-2/n^2$ that
\begin{align*}
|\delta'|\leq c|\delta|+K(f)\lambda^2 + 2\|\pi\|_2\sqrt{\log n},
\end{align*}
where we put $c\defeq 1+\epsilon_h(f)+\frac{3K(f)}{4}> 1$.
Then, applying this inequality iteratively with $t=(C/2) \log_c n$ steps, 
\begin{align*}
|\delta_t|
&\leq  c|\delta_{t-1}|+K(f)\lambda^2 + 2\|\pi\|_2\sqrt{\log n}\\
&\leq \cdots \leq c^t |\delta_0| + tc^t\left(K(f)\lambda^2 + 2\|\pi\|_2\sqrt{\log n}\right)\\
&\leq \frac{ n^{C/2}}{n^C}+ n^{C/2}\log_c n^{C/2}\left(\frac{K(f)}{n^{2C}}+\frac{2\sqrt{\log n}}{n^{C}}\right)=o(1)
\end{align*}
w.h.p., and we obtain the claim. Note that we use our assumptions of $|\delta_0|, \max{\lambda, \|\pi\|_2}\leq n^{-C}$ in the last inequality.
\end{proof}

\section[Proof of Theorem 1.7]{Proof of \cref{thm:bok}} \label{sec:bokproof}
We show \cref{thm:bok}.
The proof is almost same as the one given in \cref{sec:dynamicsofpia} but we need some special care.
We assume $k=\omega(1)$ and thus $k$ is sufficiently large.
Consider best-of-$(2k+1)$ on an $n$-vertex $\lambda$-expander graph with degree distribution $\pi$.
Suppose that the graph satisfies the conditions of \cref{thm:bok}.
Let $A_0,A_1,\ldots,$ be the sequence given by the best-of-$(2k+1)$ with initial configuration $A_0\subseteq V$.
For notational convenience, let 
\begin{align*}
&\alpha\defeq \pi(A),\, \alpha'\defeq \pi(A'),\,\alpha_t\defeq \pi(A_t), \\
&\delta\defeq \delta(A)=2\alpha-1,\, \delta'\defeq \delta(A'),\, \delta_t\defeq\delta(A_t).
\end{align*}

The dynamics of best-of-$(2k+1)$ are divided into four phases.
More specifically, we prove the following key result that corresponds to \cref{lem:mainlem}.

\begin{lemma}
\label{lem:mainlem_bok}
Consider best-of-$(2k+1)$ on an $n$-vertex $\lambda$-expander graph with degree distribution $\pi$.
Suppose that the graph satisfies the conditions of \cref{thm:bok}.
Then, the following holds:
\begin{enumerate}[label=(\Roman*)]
    \item \label{lab:phase1_bok}
    For any $A_0\subseteq V$ satisfying $|\delta_0|\leq 300C \log n/\sqrt{n}$ , $|\delta_t|\geq 300C\log n/\sqrt{n}$ within $t=O(\log n/\log k)$ steps w.h.p.
    \item \label{lab:phase2_bok}
    For any $A_0\subseteq V$ satisfying $|\delta_0|$ satisfying $300C\log n/\sqrt{n} \leq |\delta_0|\leq \frac{1.25}{\sqrt{k}}$, $|\delta_t|>\frac{1.25}{\sqrt{k}}$ within $t=O(\log n/\log k)$ steps w.h.p.
    \item \label{lab:phase3_bok}
    For any $A_0\subseteq V$ satisfying $\frac{1.25}{\sqrt{k}} \leq |\delta_0|\leq 0.9$, $|\delta_1|> 0.9$ w.h.p.
    \item \label{lab:phase4_bok}
    For any $A_0\subseteq V$ satisfying $0.9\leq |\delta_0|<1$, $|\delta_t|=1$ (or equivalently, the voting process reaches consensus) within $t=O(\log n/\log k)$ steps w.h.p.
\end{enumerate}
\end{lemma}

\begin{proof}[Proof of \cref{thm:bok} using \cref{lem:mainlem_bok}]
\Cref{thm:bok} is straightforward from
\cref{lem:mainlem_bok}.
For any initial configuration $A_0\subseteq V$, $A_0$ satisfies one of \cref{lab:phase1_bok} to \cref{lab:phase4_bok}.
If $A_0$ satisfies \cref{lab:phase4_bok}, the consensus time is $O(\log n/\log k)$.
Otherwise, from \cref{lem:mainlem_bok}, for some $t=O(\log n/\log k)$, $A_t$ satisfies $|\delta(A_t)|>0.9$ and then apply \cref{lem:mainlem_bok}\cref{lab:phase4_bok}.
\end{proof}

The rest of this section is
devoted to prove \cref{lem:mainlem_bok}.
We begin with preparing useful facts concerning with best-of-$(2k+1)$.
Let $f_{2k+1}$ be the betrayal function of best-of-$(2k+1)$.
Then, we have
\begin{align*}
    &\left|f'_{2k+1}\left(\frac{1}{2}\right)\right| = (2k+1)\binom{2k}{k}4^{-k} \geq 1.05\sqrt{k},\\
    &|f'_{2k+1}(x)| \leq \left|f'_{2k+1}\left(\frac{1}{2}\right)\right| \leq \frac{3}{\sqrt{\pi}}\sqrt{k} \leq 2\sqrt{k},\\
    &|f''_{2k+1}(x)| \leq \left|f''_{2k+1}\left(\frac{1}{2}+\frac{1}{2\sqrt{2k-1}}\right)\right| <1.6k
 \end{align*}
for sufficiently large $k$.
Here, we used $\frac{4^k}{\sqrt{\pi k}} \left(1-\frac{1}{8k}\right) \leq \binom{2k}{k} \leq \frac{4^k}{\sqrt{\pi k}}$.

From \cref{lem:expectationandvariance_symmetry} 
and \cref{lem:variance} (note that $f_{2k+1}(x)$ satisfies $f_{2k+1}(x)+f_{2k+1}(1-x)=1$), it holds for all $A\subseteq V$ that
\begin{align}
    &\bigl|\E[\alpha'] - f_{2k+1}(\alpha)\bigr| \leq 0.4k\lambda(\delta+\lambda)\alpha(1-\alpha) \label{eq:epia'bound_bok}    \\
    &\bigl|\Var[\alpha'] - g_{2k+1}(1/2)\|\pi\|_2^2\bigr| \leq
    2\sqrt{k}\left(\frac{\|\pi\|_2^2}{2}|\delta|+ \lambda \|\pi\|_3^{3/2}\right),\label{eq:vpia'bound_bok}
\end{align}
where $g_{2k+1}(x)=f(x)(1-f(x))$.
Note that $g'_{2k+1}(x)=f'_{2k+1}(x)(1-2f_{2k+1}(x))$ satisfies $|g'_{2k+1}(x)|\leq |f'_{2k+1}(x)|\leq 2\sqrt{k}$.
Thus, from the Hoeffding bound (\cref{lem:Hoeffding}), it holds w.h.p.~that
\begin{align}
    |\alpha'-f_{2k+1}(\alpha)| &\leq 0.4k\lambda^2\alpha(1-\alpha) + \|\pi\|_2\sqrt{\log n} \nonumber\\
    &\leq 0.4k\lambda^2\alpha(1-\alpha) + \sqrt{\frac{C\log n}{n}}. \label{eq:pia'bound_bok}
\end{align}
On the other hand,
it is routine to check the following facts.
\begin{align}
&\lambda\sqrt{k\|\pi\|_3^3} = o(n^{-1}), \label{eq:cond2_bok}\\
&\frac{1}{n} \leq \|\pi\|_2^2 \leq \frac{C}{n}, \label{eq:cond3_bok}\\
&k\lambda^2=O(1/\sqrt{n}) \label{eq:cond4_bok}.
\end{align}

We begin with proving the following result that corresponds to \cref{lem:delta_lower_half}.
\begin{lemma} \label{lem:delta_lower_half_bok}
There exists constants $h,c>0$ such that,
for any $A\subseteq V$ satisfying
$h/\sqrt{nk} \leq |\delta| \leq 1.25/\sqrt{k}$, 
\begin{align*}
    \Pr[|\delta'|<0.025\sqrt{k}|\delta|] \leq \exp\left(-ckn\delta^2\right).
\end{align*}
\end{lemma}
\begin{proof}
Let $h$ be a sufficiently large constant
and
let $A\subseteq V$ be a configuration satisfying $\frac{h}{\sqrt{kn}} \leq |\delta| \leq \frac{1.25}{\sqrt{k}}$.
From \cref{eq:epia'bound_bok}, \cref{eq:cond4_bok} and Taylor's theorem, we have 
\begin{align*}
|\E[\delta']| &\geq \left|2f_{2k+1}\left(\frac{1}{2}+\frac{\delta}{2}\right)-1\right|-0.4k\lambda^2\alpha(1-\alpha) \\
&\geq f'_{2k+1}\left(\frac{1}{2}\right)|\delta| - \max_{0\leq z\leq 1}|f''_{2k+1}(z)|\frac{\delta^2}{2} - 0.1k\lambda^2 \\
&\geq 0.05\sqrt{k}|\delta|+(\sqrt{k}|\delta|-0.8k\delta^2) - 0.1k\lambda^2 \\
&\geq 0.05\sqrt{k}|\delta| + \frac{0.01h}{\sqrt{n}} - 0.1k\lambda^2 \\
&\geq 0.05\sqrt{k}|\delta|. 
\end{align*}
In the fourth inequality, note that
$\sqrt{k}|\delta|\geq 0.8k\delta^2$ holds
if $|\delta|\leq 1.25/\sqrt{k}$.
In the last inequality, we used $\lambda=O(k^{-0.5}n^{-0.25})$
and thus $k\lambda^2=O(1/\sqrt{n})\leq 0.01h/\sqrt{n}$ for sufficiently large constant $h$.
Then, from \cref{cor:Hoeffding_abs}, we have
\begin{align*}
    \Pr[ |\delta'| < 0.025\sqrt{k}|\delta|]
    &\leq \Pr\left[|\delta'| < 0.5|\E[\delta']|\right] \\
    &\leq 2\exp\left(-\frac{0.5|\E[\delta']|^2}{\|\pi\|_2^2}\right) \\
    &\leq \exp\left(-ckn\delta^2\right)
\end{align*}
for some suitable constant $c>0$.
In the last inequality, we used \cref{eq:cond3_bok}.
\end{proof}

\subsection[Phase 1 (Best-of-k)]{Phase \cref{lab:phase1_bok}:~$0\leq |\delta|\leq 300C\log n/\sqrt{n}$} \label{sec:symmetrybreaking_bok}
In this part, we show \cref{lem:mainlem_bok}\cref{lab:phase1_bok}.
The proof is almost same as that of \cref{lem:mainlem}\cref{lab:phase1} that is presented in \cref{sec:symmetrybreaking}.
The difference is that we use the following result, which is a slight modification of \cref{lem:nazolemma}.

\begin{lemma}[Modification of \cref{lem:nazolemma}] \label{lem:nazolemma_bok}
Consider a Markov chain $(X_t)_{t=1}^\infty$ with finite state space $\Omega$ and a function $\Psi:\Omega \to [0,n]$.
Let $C_1$ be an arbitrary constant and $m=C_1 \sqrt{n}\log n$.
Let $k=k(n)$ be a function such that $k(n)\to\infty$ as $n\to\infty$.
Suppose that $\Omega,\Psi$ and $m$ satisfies the following conditions:

\begin{enumerate}[label=$(\roman*)$]
\item\label{state:nazo_sqrt_extend} For any positive constant $h$, there exists a positive constant $C_2<1$ such that
\begin{eqnarray*}
\Pr\left[\Psi(X_{t+1})<h\sqrt{\frac{n}{k}}\,\middle|\,\Psi(X_t)\leq m\right] < \frac{C_2}{\sqrt{k}}.
\end{eqnarray*}

\item\label{state:nazo_cher_extend} Three positive constants $C_3,C_4$ and $h$ exist such that, for any $x\in \Omega$ satisfying $h\sqrt{n/k}\leq \Psi(x)<m$,
\begin{align*}
\Pr\left[\Psi(X_{t+1})< C_3\sqrt{k}\Psi(X_t) \,\middle|\, X_t=x\right] < \exp\left(-C_4\frac{k\Psi(x)^2}{n}\right).
\end{align*}
\end{enumerate}

Then, $\Psi(X_t) \geq m$ holds w.h.p.~for some $t=O(\log n/\log k)$.
\end{lemma}

We prove \cref{lem:nazolemma_bok} in \cref{sec:proofofnazolemma_bok}.
\Cref{lem:mainlem_bok}\cref{lab:phase1_bok} is immediate from
\cref{lem:nazolemma_bok} with letting $\Psi(A)=n|\delta|$ and $C_1=300C$.
Hence, it suffices to verify the conditions \cref{state:nazo_sqrt_extend} and \cref{state:nazo_cher_extend}.

\paragraph*{Condition \cref{state:nazo_sqrt_extend}.}
First we evaluate the variance $\Var[\delta']$.
\begin{claim}
Under the same assumption as \cref{lem:mainlem_bok}\cref{lab:phase1_bok},
\begin{align*}
    \Var[\delta'] \geq \frac{0.99}{n}.
\end{align*}
\end{claim}
\begin{proof}[Proof of the claim]
Note that $\Var[\delta']=4\Var[\alpha']$.
From \cref{eq:vpia'bound_bok}, we can evaluate the variance $\Var[\alpha']$ as follows:
\begin{align*}
\Var[\alpha'] 
&\geq g_{2k+1}(1/2)\|\pi\|_2^2
-\sqrt{k}|\delta|\|\pi\|_2^2
-2\sqrt{k}\lambda\|\pi\|_3^{3/2} \\
&\geq \frac{1}{4n} -
3\sqrt{k}\left(300C^2\sqrt{\frac{\log n}{n^3}}+\lambda \|\pi\|_3^{3/2}\right) \hspace{0.5em}\text{(from \cref{eq:cond3_bok})} \\
&= \frac{1-o(1)}{4n} \hspace{.5em}\text{(since $k=o(\log n/n)$) and \cref{eq:cond2_bok})} \\
&\geq \frac{0.99}{4n}.
\end{align*}
\end{proof}
From \cref{cor:BEbound_cor2}, for any positive real $x$, we have
\begin{align*}
    \Pr\left[|\delta'|\leq x\sqrt{\frac{0.99}{n}}\right]
    \leq \Phi(x)+\frac{5.6\|\pi\|_3^3}{\Var[\delta']^{3/2}} = \Phi(x)+o(1),
\end{align*}
where $\Phi(x)=\frac{1}{\sqrt{2\pi}}\int_{-\infty}^x\mathrm{e}^{-y^2/2}\mathrm{d}y$ (see \cref{eq:biasprob}).
This yields the condition \cref{state:nazo_sqrt_extend}.

\paragraph*{Condition \cref{state:nazo_cher_extend}.}
This condition directly follows
\cref{lem:delta_lower_half_bok}
by substituting $|\delta|=\frac{\Psi(A)}{n}$.

\subsection[Phase 2 (Best-of-k)]{Phase \cref{lab:phase2_bok}:~$300C\log n/\sqrt{n}\leq |\delta|\leq 1.25/\sqrt{k}$}
Since $|\delta|\geq 300C\log n/\sqrt{n}$,
from \cref{lem:delta_lower_half_bok},
we have
\begin{align*}
    \Pr[|\delta'|<0.025\sqrt{k}|\delta|]
    \leq \exp\left(-\frac{ckn(\log n)^2}{n}\right) \leq n^{-2}
\end{align*}
if $k$ is sufficiently large.
Thus, we have $|\delta_t|\geq (0.025\sqrt{k})^t\cdot 300C\log n/\sqrt{n}$ with probability $(1-n^{-2})^{t}$.
Therefore, for some $t=O(\log n/\log k)$, $|\delta_t|\geq 1.25/\sqrt{k}$ with probability $1-n^{-1.9}$.

\subsection[Phase 3 (Best-of-k)]{Phase \cref{lab:phase3_bok}:~$1.25/\sqrt{k} < |\delta| \leq 0.9$}
We may assume that $\delta\geq 0$ without loss of generality (otherwise, consider $A^c$).
From \cref{eq:pia'bound_bok}, we have
\begin{align*}
    \delta' &\geq 2f_{2k+1}\left(\frac{1}{2}+\frac{\delta}{2}\right) -1 - 0.4k\lambda^2 - 2\sqrt{\frac{K\log n}{n}} \\
    &\geq 2f_{2k+1}\left(\frac{1}{2}+\frac{\delta}{2}\right) -1 -o(1).
 \end{align*}
We claim that $2f_{2k+1}\left(\frac{1}{2}+\frac{\delta}{2}\right)-1>0.9$ during this phase (for sufficiently large $n$ and $k$).
Let $\mathrm{Bin}(N,p)$ denote the random variable of binomial distribution with $N$ trials and probability $p$.
Then, from the definition of $f_{2k+1}$, it holds that
\begin{align}
    f_{2k+1}\left(\frac{1}{2}+\delta\right)
    &= \Pr\left[\mathrm{Bin}\left(2k+1,\frac{1}{2}+\delta\right)\geq k+1\right] \nonumber\\
    &= 1-\Pr\left[\mathrm{Bin}\left(2k+1,\frac{1}{2}+\delta\right)\leq k\right]. \label{eq:f2k+1bound1}
\end{align}
Let $\mu=(2k+1)(1/2+\delta)$ be the expectation of $\mathrm{Bin}(2k+1,1/2+\delta)$.
Then, since $\mu-k \geq 2k\delta$, we have
\begin{align}
    \Pr\left[\mathrm{Bin}\left(2k+1,\frac{1}{2}+\delta\right)\leq k\right]
    &\leq 
    \Pr\left[\mathrm{Bin}\left(2k+1,\frac{1}{2}+\delta\right)\leq \mu-(\mu-k)\right] \nonumber\\
    &\leq
    \Pr\left[\mathrm{Bin}\left(2k+1,\frac{1}{2}+\delta\right)\leq \mu-2k\delta\right] \nonumber\\   
    &\leq
    \exp(-2k\delta^2). \label{eq:f2k+1bound2}
\end{align}
In the third inequality, we applied the Hoeffding bound (\cref{lem:Hoeffding}).
If $\delta \geq \frac{1.25}{\sqrt{k}}$, by combining \cref{eq:f2k+1bound1,eq:f2k+1bound2}, we obtain
\begin{align*}
    f_{2k+1}\left(\frac{1}{2}+\delta\right) &\geq 1-\exp(-2k\delta^2) \\
    &\geq 1-\mathrm{e}^{-3.125} \\
    &> 0.92.
\end{align*}
Thus, from \cref{eq:pia'bound_bok},
$\delta'\geq 0.92-o(1) > 0.9$ holds w.h.p.

\subsection[Phase 4 (Best-of-k)]{Phase \cref{lab:phase4_bok}:~$0.9 < |\delta| \leq 1$}
We may assume $\pi(A_0)\leq 0.1$ without loss of generality.
We claim that $\pi(A_t)<\frac{1}{n^2}$ for some $t=O(\log n/\log k)$, which implies $A_t=\emptyset$ (since $\pi(S)\geq \frac{1}{2m} \geq \frac{1}{n^2}$ whenever $S\neq \emptyset$).

Observe that
\begin{align*}
    f_{2k+1}(x) &= \sum_{i=k+1}^{2k+1} \binom{2k+1}{i}x^i(1-x)^{2k+1-i} \\
    &\leq (4x)^k \\
    &\leq \frac{x}{4k}
\end{align*}
whenever $x\leq 0.1\leq 4^{-1}(16k)^{1/(k-1)}$ with $k\geq 2$.
Therefore, from \cref{eq:epia'bound_bok}, we have
\begin{align*}
    \E[\alpha'] &\leq \left(\frac{1}{4k} + 0.4k\lambda^2\right)\alpha.
\end{align*}
From \cref{eq:pia'bound_bok} and the upper bound of $\E[\alpha']$ above,
it holds with probability $1-O(n^{-3})$
that $\alpha'\leq 0.9$ conditioned on $\alpha\leq 0.1$.
Thus, 
\begin{align*}
    \E[\pi(A_t)] \leq \left(\frac{1}{4k} + 0.4k\lambda^2\right)^t + n^{-3+o(1)} \leq n^{-3+o(1)}
\end{align*}
for some $t=O(\log n/\log k + \log n/\log \lambda^{-1})=O(\log n/\log k)$ (note that $\lambda^{-1}=\Omega(n^{1/4})$ from \cref{eq:cond4_bok}).
For this $t$, we have $\Pr[A_t\neq\emptyset] \leq \Pr[\pi(A_t)\geq n^{-2}]\leq  n^2\E[\pi(A_t)]=O(n^{-1})$.
This completes the proof of \cref{lem:mainlem_bok} as well as
\cref{thm:bok}.

\subsection[Proof of Lemma 6.3]{Proof of \cref{lem:nazolemma_bok}} \label{sec:proofofnazolemma_bok}
The proof is essentially given in \cite{CGGNPS18}.
By
inspecting the proof of \cite{CGGNPS18} with evaluating constant terms carefully,
we obtain \cref{lem:nazolemma_bok}.
For completeness, let us present it here.

Let $m=C_1\sqrt{n}\log n$.
Let $\tau=\inf\{ t\in \mathbb{N}:\Psi(X_t)\geq m\}$ and $\{\tau(i)\}_{i\in\mathbb{N}}$ be the hitting times defined as
\begin{align*}
    \begin{cases}
    \tau(0) = 0,\\
    \tau(i) = \inf_{t\in \mathbb{N}}\{t:\tau(i-1)<t< \tau,f(X_t)\geq h\sqrt{n/k}\}.
    \end{cases}
\end{align*}
Let $R_1,R_2,\ldots$ be the sequence of random variables defined as $R_i=X_{\tau(i)}$.
It is shown in \cite{CGGNPS18} that
\begin{itemize}
    \item The sequence $(R_i)_{i\in\mathbb{N}}$ is a Markov chain.
    \item The sequence $(R_i)_{i\in\mathbb{N}}$ satisfies
    \begin{align*}
        \Pr[\Psi(R_{i+1}) < C_3\sqrt{k}\Psi(R_i)|R_i=x] < \exp\left(-C_4\frac{k\Psi(x)^2}{n}\right)
    \end{align*}
    for any $x\in\Omega$ that $h\sqrt{n/k}\leq \Phi(x) < m$.
\end{itemize}

We claim that $\Psi(R_i)\geq m$ for some $i=O(\log n/\log k)$.
To prove this, we use the Markov inequality.
Fix a state $x\in \Omega$ such that $h\sqrt{n/k}\leq \Psi(x)<m$ for a sufficiently large constant $h$.
Let $Y_i=\exp(-\frac{\Psi(R_i)}{\sqrt{n}})$ for each $i$.
Let $y=\exp(-\frac{\Psi(x)}{\sqrt{n}})$ and $z=z(x)=\frac{\sqrt{k}\Psi(x)}{\sqrt{n}}\geq h$ for $x\in \Omega$.
Note that $\mathrm{e}^{z} = y^{-\sqrt{k}}$.
Then, we have
\begin{align*}
    &\E[Y_{i+1}|R_i=x] \\
    &\leq \Pr[\Psi(R_{i+1})<C_3\sqrt{k}\Psi(x)]+\Pr[\Psi(R_{i+1})\geq C_3\sqrt{k}\Psi(x)]\cdot \exp\left(-C_3\sqrt{k}\frac{\Psi(x)}{\sqrt{n}}\right) \\
    &\leq \exp\left(-C_4\frac{k\Psi(x)^2}{n}\right) + \exp\left(-C_3\frac{\sqrt{k}\Psi(x)}{\sqrt{n}}\right) \\
    &= \exp\left(-C_4z^2\right)+\exp\left(-C_3z\right)\\
    &= y^{-\frac{C_3}{2}\sqrt{k}}
    \left(\exp\left(\frac{C_3}{2}z-C_4z^2\right)+\exp\left(-\frac{C_3}{2}z\right)\right) \\
    &\leq \frac{1}{2}y^{\frac{C_3}{2}\sqrt{k}} \hspace{2em}\left(\text{since $z \geq h$ is sufficiently large and $C_2>1$}\right) \\
    &\leq \begin{cases}
    \frac{1}{2} & \text{if $\frac{1}{2}< y_i\leq 1$},\\
    \frac{y}{C_3\sqrt{k}} & \text{if $y_i\leq\frac{1}{2}$}.
    \end{cases}
\end{align*}
In the second part of the last inequality, we assume that $k\geq 2$; hence, it holds that $r^a \leq \frac{r}{a}$ for $0\leq r\leq \frac{1}{2}$ if $a\geq 2$.
Note that for each $i\geq 1$, the random variable $\Psi(R_i)=\Psi(X_{\tau(i)})$ satisfies $h\sqrt{n/k}\leq \Psi(R_i) < m$.
Then, we have
\begin{align*}
    \E[Y_i] \leq \frac{1}{2} \left(\frac{1}{C_3\sqrt{k}}\right)^{i-2}
\end{align*}
and thus, by the Markov inequality,
\begin{align*}
    \Pr[\Psi(R_i)<m] &= \Pr\left[Y_i > \exp\left(-\frac{m}{\sqrt{n}}\right)\right] \\
    &\leq \exp\left(\frac{m}{\sqrt{n}}\right)
    \frac{1}{2}\left(\frac{1}{C_3\sqrt{k}}\right)^{i-2} \\
    &= \frac{n^{C_1}}{2(C_3\sqrt{k})^{i-2}}\\
    &\leq n^{-1}
\end{align*}
for $i=\lfloor C_5 \log n/\log k\rfloor$ for some constant $C_5$ that depends on $C_1$ and $C_3$.

Finally, we consider $\tau(\lfloor C_5\log n / \log k \rfloor)$.
Let $W_0,W_1,\ldots$ be binary random variables defined as
\begin{align*}
    W_t &= \begin{cases}
        1 & \text{if $\Psi(X_t)\geq h\sqrt{\frac{n}{k}}$},\\
        0 & \text{otherwise}.
    \end{cases}
\end{align*}
Note that $\Pr[\tau(T_1)\geq T_2] = \Pr[\sum_{t=1}^{T_2} W_t \leq T_1]$.
Let $\hat{W}_0,\hat{W}_1,\ldots$ be i.i.d.~binary random variables such that $\E[\hat{W}_t]=  1-\frac{C_1}{\sqrt{k}}$.
From the condition \cref{state:nazo_sqrt_extend}, for every $T$, the sum $\sum_{t=1}^T \hat{W}_t$ has stochastic dominance over $\sum_{t=1}^T W_t$.
Therefore, setting $T_1=\lfloor \frac{C_4\log n}{\log k} \rfloor$ and $T_2=\lceil \frac{2C_4\log n}{\log k}\rceil$, we obtain
\begin{align*}
    \Pr\left[\tau\left(\left\lfloor \frac{C_5\log n}{\log k} \right\rfloor\right)\geq T_2\right]
    &= \Pr\left[\sum_{t=1}^{T_2} W_t \leq \left\lfloor \frac{C_5\log n}{\log k} \right\rfloor\right] \\
    &\leq \Pr\left[\sum_{t=1}^{T_2} W_t \leq \frac{C_5\log n }{\log k} \right] \\
    &\leq \Pr\left[\sum_{t=1}^{T_2} \hat{W}_t \leq \frac{C_5\log n }{\log k} \right] \\ &\leq \Pr\left[\sum_{t=1}^{T_2} (1-\hat{W}_t) \geq T_2-\frac{C_5\log n}{\log k}\right] \\
    &\leq \Pr\left[\sum_{t=1}^{T_2}(1-\hat{W}_t) \geq \frac{C_5\log n}{\log k}\right] \\
    &\leq 2^{T_2}\left(\frac{C_1}{\sqrt{k}}\right)^{\frac{C_5\log n}{\log k}} \\
    &\leq n^{O(1/\log k)-\frac{C_5}{2}}.
\end{align*}
In the fifth inequality, we used the union bound over the choice for $\hat{W}_t$.
Note that $1-\hat{W}_t=1$ with probability $\frac{C_1}{\sqrt{k}}$.\hspace{\fill}\qedsymbol

\section{Conclusion}
In this paper we propose \fvoting as a generalization of several known voting processes.
We show that the consensus time is $O(\log n)$ for any \qmaj \fvoting on $O(n^{-1/2})$-expander graphs with balanced degree distributions.
This result extends previous works concerning voting processes on expander graphs.
Possible future direction of this work includes
\begin{enumerate}
    \item Does $O(\log n)$ worst-case consensus time holds for \qmaj \fvoting on graphs with less expansion (i.e., $\lambda=\omega(n^{-1/2})$)?
    \item Is there some relationship between \bk and \Maj?
\end{enumerate}

\section*{Acknowledgements}
This work is supported by JSPS KAKENHI Grant
Number 19J12876 and 19K20214, Japan.


\bibliographystyle{abbrv}
\bibliography{ref}
\appendix

\section{Tools}

\begin{lemma}[Lemma 3 of \cite{CRRS17}]
\label{lem:square}
Suppose that $P$ is reversible.
Then, for any $S\subseteq V$,
\begin{align*}
    \left|\sum_{v\in V}\pi(v)P(v,S)^2-\pi(S)^2\right|\leq \lambda^2\pi(S)\bigl(1-\pi(S)\bigr).
\end{align*}
\end{lemma}
\begin{corollary}
\label{lem:second_approx_lambda}
Suppose that $P$ is reversible. 
Then, for any $S\subseteq V$, 
\begin{align*}
\sum_{v\in V}\pi(v)\bigl(P(v,S)-\pi(S)\bigr)^2
\leq \lambda^2\pi(S)\bigl(1-\pi(S)\bigr). 
\end{align*}
\end{corollary}
\begin{proof}
Since $Q(V,S)=Q(S,V)=\pi(S)$ for any reversible $P$ and $S\subseteq V$, we have
\begin{align*}
\sum_{v\in V}\pi(v)\bigl(P(v,S)-\pi(S)\bigr)^2
&=\sum_{v\in V}\pi(v)P(v,S)^2+\pi(S)^2-2\pi(S)Q(V,S) \\
&=\sum_{v\in V}\pi(v)P(v,S)^2-\pi(S)^2
\leq \lambda^2\pi(S)\bigl(1-\pi(S)\bigr). 
\end{align*}
Here, we invoked \cref{lem:square} in the last inequality.
\end{proof}

\begin{lemma}[The Hoeffding bound (see, e.g.,~Theorem 10.9 of \cite{Doerr18})]
\label{lem:Hoeffding}
Let $Y_1, Y_2, \ldots, Y_n$ be independent random variables.
Assume that each $Y_i$ takes values in a real interval $[a_i, b_i]$ of length $c_i\defeq b_i-a_i$.
Let $Y=\sum_{i=1}^nY_i$.  
Then, for any $\kappa>0$,
\begin{align*}
\Pr\left[Y\geq \E[Y]+\kappa \right]&\leq \exp\left(-\frac{2\kappa^2 }{\sum_{i=1}^nc_i^2}\right), \\
\Pr\left[Y\leq \E[Y]-\kappa \right]&\leq \exp\left(-\frac{2\kappa^2 }{\sum_{i=1}^nc_i^2}\right).
\end{align*}
\end{lemma}
\begin{corollary}
\label{cor:Hoeffding_abs}
Let $Y_1, Y_2, \ldots, Y_n$ be independent random variables.
Assume that each $Y_i$ takes values in a real interval $[a_i, b_i]$ of length $c_i\defeq b_i-a_i$.
Let $Y=\sum_{i=1}^nY_i$.  
Then, for any $\kappa>0$,
\begin{align*}
\Pr\left[|Y|\geq \left|\E[Y]\right|+\kappa\right]&\leq 2\exp\left(-\frac{2\kappa^2}{\sum_{i=1}^nc_i^2}\right), \\
\Pr\left[|Y|\leq \bigl|\E[Y]\bigr|-\kappa\right] &\leq 2\exp\left(-\frac{2\kappa^2}{\sum_{i=1}^nc_i^2}\right).
\end{align*}
\end{corollary}
\begin{proof}
For the first inequality, it is straightforward to see that
\begin{align*}
\Pr\left[|Y|\geq \left|\E[Y]\right|+\kappa\right]
&=\Pr\left[|Y|-\left|\E[Y]\right|\geq \kappa\right]
\leq \Pr\left[\left|Y-\E[Y]\right|\geq\kappa\right]\\
&\leq 2\exp\left(-\frac{2\kappa^2}{\sum_{i=1}^nc_i^2}\right).
\end{align*}
Note that $|x|-|y|\leq |x-y|$ for any $x,y\in \mathbb{R}$. 
Similarly, it holds that
\begin{align*}
\Pr\left[|Y|\leq \left|\E[Y]\right|-\kappa\right]
&=\Pr\left[\left|\E[Y]\right|-|Y|\geq \kappa\right]
\leq \Pr\left[\left|\E[Y]-Y\right|\geq \kappa\right] \\
&\leq 2\exp\left(-\frac{2\kappa^2}{\sum_{i=1}^nc_i^2}\right),
\end{align*}
and we obtain the claim.
\end{proof}
\begin{lemma}[Berry-Esseen theorem (see, e.g.,~\cite{Shevtsova10})]
\label{lem:Berry-Esseen}
Let $Y_1, Y_2, \ldots, Y_n$ be independent random variables such that
$\E[Y_i]=0$, $\E[Y_i^2]>0$, $\E[|Y_i|^3]<\infty$ for all $i\in [n]$, and $\sum_{i=1}^n\E[Y_i^2]=1$.
Let $Y=\sum_{i=1}^nY_i$ and 
$\Phi(x)=\frac{1}{\sqrt{2\pi}}\int_{-\infty}^x\mathrm{e}^{-y^2/2}\mathrm{d}y$ (the cumulative distribution function of the standard normal distribution).
Then
\begin{align*}
\sup_{x\in \mathbbm{R}}\bigl|\Pr\left[Y\leq x\right]-\Phi(x)\bigr|
&\leq 5.6 \sum_{i=1}^n\E[|Y_i|^3].
\end{align*}
\end{lemma}
\begin{corollary} \label{cor:BEbound_cor2}
Let $Y_1, Y_2,\ldots, Y_n$ be independent random variables, $c=(c_1,\ldots,c_n)\in\mathbb{R}^n$ be a vector, and $Y=\sum_{i=1}^nY_i$.
Suppose that, for all $i\in [n]$, $|Y_i-\E[Y_i]|\leq c_i<\infty$ and $\Var[Y]> 0$.
Let $\Phi(x)=\frac{1}{\sqrt{2\pi}}\int_{-\infty}^x\mathrm{e}^{-y^2/2}\mathrm{d}y$. 
Then, for any positive $x\in \mathbb{R}$,
\begin{align*}
\Pr\left[|Y|\leq x\sqrt{\Var[Y]}\right]
&\leq \Phi(x)+\frac{5.6\|c\|_3^3}{\Var[Y]^{3/2}}.
\end{align*}
\end{corollary}
\begin{proof}
For each $i\in [n]$, let
\begin{align*}
Z_i\defeq \frac{Y_i-\E[Y_i]}{\sqrt{\Var[Y]}}, \hspace{1em}
Z\defeq \sum_{\substack{i\in [n]: \E[Z_i^2]>0}}Z_i
=\sum_{i\in [n]}Z_i.
\end{align*}
Note that 
$\E[Z_i^2]=0
\iff \sum_{z}z^2\Pr[Z_i=z]=0
\iff \Pr[Z_i=0]=1$.
For all $i\in \{j\in [n]: \E[Z_j^2]> 0\}$, 
it is easy to check that
$\E[Z_i]=0$, 
$\E[Z_i^2]>0$, and 
$\E[|Z_i|^3]\leq \frac{c_i^3}{\Var[Y]^{3/2}}<\infty$.
Furthermore,  
\begin{align*}
\sum_{\substack{i\in [n]:\E[Z_i^2]>0}}\E[Z_i^2]
&=\sum_{i\in [n]}\E[Z_i^2]
=\frac{\sum_{i\in [n]}\E[(Y_i-\E[Y_i])^2]}{\Var[Y]}=1.
\end{align*}
Thus, we can apply \cref{lem:Berry-Esseen} to $Z$ and it holds that
\begin{align}
\left|\Pr\left[\frac{Y-\E[Y]}{\sqrt{\Var[Y]}} \leq x\right]-\Phi(x)\right|
&=\left|\Pr\left[\sum_{i=1}^n Z_i \leq x\right]-\Phi(x)\right| 
=\left|\Pr\left[Z \leq x\right]-\Phi(x)\right| \nonumber\\
&\leq 5.6 \sum_{i\in[n]:\E[Z_i^2]>0}\E[|Z_i|^3] \nonumber\\
&\leq 5.6 \sum_{i=1}^n\frac{c_i^3}{\Var[Y]^{3/2}}
=\frac{5.6\|c\|_3^3}{\Var[Y]^{3/2}}.
\label{eq:BEbound_cor}
\end{align}
Next we observe that
\begin{align}
\Pr\left[|Y|\geq x\sqrt{\Var[Y]}\right]
&=\Pr\left[Y\geq x\sqrt{\Var[Y]}\right]+\Pr\left[Y\leq -x\sqrt{\Var[Y]}\right]
\label{eq:corBE1}
\end{align}
holds. If $\E[Y]\geq 0$, we have
\begin{align*}
\Pr\left[|Y|\geq x\sqrt{\Var[Y]}\right]
&\geq \Pr\left[Y\geq x\sqrt{\Var[Y]}+\E[Y]\right] \\
&\geq 1-\Pr\left[Y-\E[Y]\leq x\sqrt{\Var[Y]}\right]
\geq 1-\Phi(x)-\frac{5.6\|c\|_3^3}{\Var[Y]^{3/2}} 
\end{align*}
from \cref{eq:BEbound_cor}. 
Similarly, if $\E[Y]\leq 0$, \cref{eq:BEbound_cor} yields
\begin{align*}
\Pr\left[|Y|\geq x\sqrt{\Var[Y]}\right]
&\geq 
\Pr\left[Y\leq -x\sqrt{\Var[Y]}+\E[Y]\right]\nonumber \\
&= \Pr\left[Y-\E[Y]\leq -x\sqrt{\Var[Y]}\right]
\geq \Phi(-x)-\frac{5.6\|c\|_3^3}{\Var[Y]^{3/2}}. 
\end{align*}
Thus, the claim holds for both cases. 
Note that $\Phi(-x)=1-\Phi(x)$ holds.
\end{proof}
\end{document}